\documentclass[letterpaper]{amsart}
\usepackage{amsthm}
\usepackage{mathtools,mathrsfs,amssymb,latexsym,graphics,enumerate}
\usepackage[mathscr]{eucal}
\usepackage{amsmath,amsfonts,amsthm,amssymb}
\usepackage{color}
\usepackage{hyperref}
\usepackage[numbers,sort&compress]{natbib}
\usepackage[left=1in,right=1in,top=1in,bottom=1in]{geometry}

\numberwithin{equation}{section}

\newcommand{\rr}{\mathbb{R}}

\newcommand{\lan}{\langle}
\newcommand{\ran}{\rangle}
\newcommand{\be}{\begin{eqnarray*}}
\newcommand{\bel}{\begin{eqnarray}}
\newcommand{\ee}{\end{eqnarray*}}
\newcommand{\eel}{\end{eqnarray}}
\newcommand{\ba}{\begin{aligned}}
\newcommand{\ea}{\end{aligned}}
\newcommand{\de}{\Delta}
\newcommand{\al}{\alpha}
\newcommand{\na}{\nabla}
\newcommand{\ep}{\epsilon}
\newcommand{\ra}{\rightarrow}
\newcommand{\pa}{\partial}
\newcommand{\nb}{}
\newcommand{\nq}{{\neq}}
\newcommand{\myb}[1]{{ #1 }}
\newcommand{\myr}[1]{}
\newcommand{\wh}{\widehat}

\newcommand{\pay}{\partial_{y}}

\newcommand{\om}{\Omega}

\newtheorem{theorem}{Theorem}

\newtheorem{lem}{Lemma}
\newtheorem{pro}{Proposition}
\newtheorem{rmk}{Remark}

\numberwithin{theorem}{section}
\numberwithin{cor}{section}
\numberwithin{pro}{section}
\numberwithin{rmk}{section}
\numberwithin{lem}{section}

\mathtoolsset{showonlyrefs=true}
\newcommand{\norm}[1]{\left\lVert#1\right\rVert}

\newcommand\Torus{{\mathbb T}}

\title[Enhanced Dissipation and Blow-up Suppression]{Enhanced dissipation and blow-up suppression in a chemotaxis-fluid system}
\date{\today}
\author{Siming He}
\address{Department of Mathematics, Duke University, Durham, NC 27708
\& Department of Mathematics, University of South Carolina, Columbia, SC, 29208}
\email{simhe@math.duke.edu, siming@mailbox.sc.edu}

\thanks{\textbf{Acknowledgment.} This work was supported in part by NSF grants DMS 2006660, DMS 2304392, DMS 2006372. The author would like to thank Jacob Bedrossian, Alexander Kiselev, Eitan Tadmor and Fei Wang for many helpful discussions. The author would also like to thank Kevin Hu and anonymous referees for pointing out typos in the first version of the manuscript and providing many insightful suggestions.}

\begin{document}
\begin{abstract}
In this paper, we investigate a coupled Patlak-Keller-Segel-Navier-Stokes (PKS-NS) system. We show that globally regular solutions with arbitrary large cell populations exist. The primary blowup suppression mechanism is the shear flow mixing induced enhanced dissipation phenomena.
\end{abstract}
\maketitle
\setcounter{tocdepth}{1}{\tableofcontents}

\section{Introduction}
We consider the coupled Patlak-Keller-Segel-Navier-Stokes (PKS-NS) system modeling the chemotaxis phenomenon in a moving fluid:
\begin{align}\label{pePKS-NS}
\left\{\begin{array}{rrrrr}\ba
\pa_t n\ +\ &v\cdot \na n\ +\ \kappa  \na \cdot( n\na c)\ =\ \kappa  \de n,\ \ (1-\de) c\ = \ n,\\
\pa_t v\ +\ &(v\cdot \na )v\ +\ \na p\ =\ \nu  \de v\  +\ \kappa   n\na c,\quad\na\cdot v\ =\ 0,\\
n(t=&0 )=n_{\text{in}} ,\quad v(t=0 )=v_{\text{in}} , \quad (x,y)\in  \Torus\times\rr.\ea\end{array}\right.
\end{align}
Here $n$ denotes the cell density, and $c$ is the chemoattractant density. The divergence-free vector field $v$ indicates the ambient fluid velocity. The first equation (Patlak-Keller-Segel equation) describes the time evolution of the cell density subject to transportation by ambient fluid flow $v$, aggregation trigged by chemotaxis, and diffusion in the media. The aggregation and diffusion take effects on a time scale $\mathcal{O}(\kappa^{-1}),\,\kappa\in(0,\infty)$. The cells move towards higher concentrations of the chemoattractants. In the meantime, they secrete the chemoattractants to re-enhance this aggregation effect. By assuming that the secretion and redistribution of chemoattractants happen at a fast timescale, we establish an elliptic-type partial differential relation between the density distributions, $n$ and $c$. The equation (Navier-Stokes equation) on the divergence-free vector field $v$ describes the fluid motion subject to force. The parameter $\nu$ is the inverse Reynold number, and the scalar function $p$ denotes the pressure that ensures the divergence-free condition. The fluid exerts friction force on the moving cells to guarantee that they move without acceleration. Hence, Newton's law predicts that there exists a reaction force from the cells to the fluid. The coupling $n \na c$ in the {N}avier-{S}tokes equation models this interaction. The same forcing appears in the {N}ernst-{P}lanck-{N}avier-{S}tokes system, see, e.g., \cite{ConstantinIgnatova19}.

If no ambient fluid flows are present, i.e., $v\equiv 0$, the coupled system \eqref{pePKS-NS} simplifies to a variant of the classical parabolic-elliptic Patlak-Keller-Segel (PKS) equation
\begin{align}\pa_t n+\kappa\na\cdot(n \na c)=\kappa \de n,\quad -\de c=n.\label{PKS}
\end{align} 
The equation \eqref{PKS} is derived by C. Patlak \cite{Patlak}, and E. Keller and L. Segel \cite{KS}. Simplified models are proposed by V. Nanjundiah, \cite{Nanjundiah1973}. The literature on the PKS model is extensive, and we refer the interested readers to the representative works, \cite{JagerLuckhaus92, Nagai95, NagaiSenbaYoshida97,
Biler06,Hortsmann,BlanchetEJDE06,CalvezCorrias,BlanchetCarrilloMasmoudi08,CarrilloRosado10,  BMKS13,Bedrossian15, EganaMischler16}, and the references therein. We summarize the main results on the plane $\rr^2$ as follows. Thanks to the divergence structure of the PKS equations \eqref{PKS}, the total cell population/mass is conserved over time, i.e., $M:=\|n(t)\|_{L^1}=\|n_{\text{in}}\|_{L^1}$. The long-time behavior of the equation \eqref{PKS} hinges on the total mass $M$. Suppose the initial cell density has a finite second moment and a total mass $M$ strictly less than $8\pi$. In that case, the unique solutions to \eqref{PKS} become smooth instantly and exist for all time, see, e.g.,  \cite{BlanchetEJDE06,CarrilloRosado10,BMKS13, EganaMischler16,Wei181,HsiehYu20}.  A key observation in deriving sharp regularity results is that the Patlak-Keller-Segel equations have natural dissipative free energy
\begin{align}
E:=\int n\log n -\frac{1}{2}nc dV. 
\end{align}
On the contrary, if the initial mass is strictly larger than $8\pi$, the solutions with finite second moment blow up in finite time, e.g., \cite{JagerLuckhaus92,CalvezCarrillo06,BlanchetEJDE06,CalvezCorrias}. The refined description of the singularities is provided in work \cite{HerreroVelazquez96,Velazquez02,Velazquez04I,Velazquez04II,RaphaelSchweyer14,CollotGhoulMasmoudiNguyen19, CollotGhoulMasmoudiNguyen192}. In the borderline case, $M=8\pi$, the solutions with finite second moments form Dirac-type singularities as the time approaches infinity, \cite{BlanchetCarrilloMasmoudi08, GhoulMasmoudi18,DaviladelPinoDolbeaultMussoWei19}.

If there is ambient fluid flow, the long-time dynamics of the systems \eqref{pePKS-NS} are delicate. In the pioneering work, \cite{KiselevXu15}, A. Kiselev and X. Xu show that if the fluid vector field $v$ is relaxation enhancing in the sense of \cite{ConstantinEtAl08}, then by choosing a large enough amplitude ($\|v\|_\infty$), the chemotactic blowups are suppressed. Their analysis is later generalized in \cite{IyerXuZlatos} to a broader class of fluid vector fields. Furthermore, in work, \cite{BedrossianHe16, He}, the authors show that strong shear flows can suppress the blowups through a fast dimension reduction process. Last but not least, the authors of \cite{HeTadmor172, HeTadmorZlatos21} exploit the fast-spreading scenario of the hyperbolic and quenching shear flows to reach the same goal. In work mentioned above, the ambient fluid velocity fields $v$ are passive because the dynamics of the cell density do not alter the fluid itself. If there is active coupling between the cell dynamics and the fluid motion, the only known result is \cite{ZengZhangZi21}. In this work, the authors prove that if the underlying fluid flow is close to the Couette flow, strong enough shear suppresses the blowup of a specific type of PKS-NS system.

One can regard the system \eqref{pePKS-NS} as one among many attempts to model the chemotaxis phenomena in a fluid. The literature on coupled chemotaxis-fluid equations is vast, and we refer the interested readers to the papers \cite{Lorz10,Lorz12,LiuLorz11,DuanLorzMarkowich10, FrancescoLorzMarkowich10,Tuval05,Winkler12,TaoWinkler,ChaeKangLee13,KozonoMiuraSugiyama,Winkler20} and the references therein. Many of these works investigate coupled systems involving fully parabolic Patlak-Keller-Segel and Navier-Stokes equations. The parabolic nature of the chemical equations complicates the analysis. For example,  I. Tuval et al.  proposed the following model \cite{Tuval05},
\begin{align}
\left\{\begin{array}{rrrrr}\ba
\pa_t n+&v\cdot \na n+\na \cdot( n\na c)=\de n,\\
\pa_t c+&v\cdot \na c=\de c-n f(c),\\
\pa_t v+&(v\cdot \na )v+\na p=\de v +n\na \phi,\quad\na\cdot v=0.\ea\end{array}\right.
\end{align}
Here a parabolic equation governs the dynamics of the chemicals (oxygen), and the coupling $n\na \phi$ in the fluid equation models the buoyancy effect. The regularity and long-time behaviors of the system are explored in \cite{Winkler14, Winkler16, Winkler17, Winkler21}. 
On the other hand, A. Lorz \cite{Lorz12}, and H. Kozono et al. \cite{KozonoMiuraSugiyama} proposed models whose chemical densities are determined through elliptic-type relations.

A new feature of the coupled system \eqref{pePKS-NS} is that it retains dissipative free energy 
\begin{align*}
F=\int n\log n  -\frac{1}{2}  nc  +\frac{1}{2}  |v|^2 dV. 
\end{align*} 
The dissipative free energy, together with the logarithmic Hardy-Littlewood-Sobolev inequality, yields global regularity of the solutions to a variant of \eqref{pePKS-NS} in the entire subcritical mass regime, i.e., $M<8\pi$ (\cite{GongHe20}). The critical mass case is analyzed in \cite{LaiWeiZhou21}. In the supercritical case $M>8\pi$, there exists a solution with finite-time blow-up (\cite{GongHe20}). 
 
We consider the system perturbed around the Couette flow $v(x,y)=(y,0)$, a stationary solution to the Naver-Stokes equation. By decomposing the velocity field as $v=y + {u}$ and writing the fluid equation in vorticity form, we end up with the system:
\begin{align}
\left\{\begin{array}{rrrrr}\ba
\pa_t n+&y\pa_x  n+ {u}\cdot \na n+\kappa\na \cdot( n\na c)=\kappa\de n,\ \ (1-\de)  c=n,\\
\pa_t  {\omega}+&y\pa_x {\omega}+ {u}\cdot \na  {\omega} =\nu\de {\omega} +\kappa \na^\perp\cdot (n\na c),\quad  {u}=\na^\perp \de ^{-1} {\omega},\\
n(t=&0)=n_{\text{in}},\quad \omega(t=0 )=\omega_{\text{in}}, \quad (x,y)\in  \Torus\times\rr.\ea\end{array}\right.\label{pePKS-NS-Couette}
\end{align}
Here ${\omega=\na^\perp\cdot u=-\pa_{y}u^{(1)}+\pa_xu^{(2)}}\,$ and $\, \na^\perp=(-\pa_y,\pa_x)$. In this formulation, we can view the coupled system as a nonlinear perturbation of the Naver-Stokes solution $(y,0)$. The problem of suppressing the chemotactic blow-up is now equivalent to a nonlinear stability problem of the Couette flow. The fundamental question in the study of hydrodynamic stability is to determine the functional space $X$ and the parameter $\al\in[0,\infty)$ such that 
\begin{align}
\|\omega_{\text{in}}\|_{X}\lesssim \nu^\al\quad \Rightarrow \quad \text{Stability}.
\end{align}
Here $\nu^\al$ is the stability threshold of the flow associated with the spaces $X$. The nonlinear stability threshold of the Couette flow has attracted much attention in the last decade, see, e.g., \cite{BMV14, BGM15I, BGM15II, BGM15III, BVW16, ChenLiWeiZhang18, MasmoudiZhao22, MasmoudiZhao20}. A complete survey of the literature is out of the scope of this paper. Therefore, we highlight some of the work focusing on a 2-dimensional setting. In \cite{BMV14}, J. Bedrossian, N. Masmoudi, and V. Vicol showed that the stability threshold of the $2D$-Couette flow is $\mathcal{O}(1)$ in the Gevrey spaces. The stability thresholds associated with the Sobolev spaces are shown to be  $\mathcal{O}(\nu^{1/2})$ on the cylinder $\Torus\times \rr$ (J. Bedrossian, V. Vicol and F. Wang, \cite{BVW16}) and in the channel (Q. Chen, T. Li, D. Wei, and Z. Zhang, \cite{ChenLiWeiZhang18}). Later, N. Masmoudi and W. Zhao considered higher-order Sobolev norms ($H^{49}$) and improved the threshold to $\mathcal{O}(\nu^{1/3})$, \cite{MasmoudiZhao22}. The enhanced dissipation phenomenon of the Couette flow plays an essential role in the above works. We also refer the interested readers to the work \cite{Liss20, WenWuZhang21, MasmoudiSaidHouariZhao20, Zillinger21, Zillinger212} for the detailed stability analysis of Couette flow in MHD, Boussinesq equations. For the enhanced dissipation phenomena associated with other shear flows, we refer the interested readers to the work \cite{AlbrittonBeekieNovack21, BCZ15, Wei18, He21, WeiZhangZhao20, LiWeiZhang20, LiZhao21, CotiZelatiDrivas19}, and the references therein.

Our main result is that global-in-time regular solutions with arbitrarily large mass $M$ to the system \eqref{pePKS-NS} exist.
\begin{theorem}\label{thm:main}
Consider the equation \eqref{pePKS-NS-Couette} subject to  initial conditions ${n_\mathrm{in}\geq0},\ n_\mathrm{in}\in L^1\cap H^s(\Torus\times\rr)$, ${n_\mathrm{in}|y|^2\in L^1(\Torus\times\rr)}$,  $\omega_{\mathrm{in}}\in H^s(\Torus\times\rr), \, { 5\leq s\in\mathbb{N}}$. Assume that the parameters $\kappa, \ \nu$ are in the regime $ 0<\kappa\leq\nu\leq1$. There exists a threshold $\ep_0(\|n_\mathrm{in}\|_{L^1\cap H^s})\in(0,1) $ such that if the following relations hold
\begin{align}\label{initial_data}
\|\omega_{\mathrm{in}}\|_{H^s}\leq \ep\nu^{1/2}, \quad \kappa= \ep\nu,\quad 0<\ep\leq \ep_0, 
\end{align}
the regular solutions to \eqref{pePKS-NS-Couette} exist for all time. \myb{Moreover, there exists a universal constant $\delta\in(0,1)$ such that the following  enhanced dissipation estimate holds 
 \begin{align}\label{ED}
 \left\|e^{\delta \kappa^{1/3}|\pa_x|^{2/3}t}\left(n-\frac{1}{|\Torus|}\int_\Torus n dx\right)\right\|_{L^2} +\left\|e^{\delta \kappa^{1/3}|\pa_x|^{2/3}t}\left(\omega-\frac{1}{|\Torus|}\int_\Torus \omega dx\right)\right\|_{L^2}\leq \mathcal{B}(\|n_{\mathrm{in}}\|_{L^1\cap H^s}), \quad \forall t\in[0,\infty).
 \end{align}}
{Here the bound $\mathcal{B}$ only depends on the initial data.} 
\end{theorem}          

\begin{rmk}
We compare our result with that of \cite{ZengZhangZi21}. In \cite{ZengZhangZi21}, the authors considered a similar system with buoyancy force coupling between the fluid and cell-density equations and showed suppression of blowup results. However, here our system \eqref{pePKS-NS} involves a nonlinear coupling in the fluid equation, which complicates the analysis. Moreover, the methods we employ here are different from that of \cite{ZengZhangZi21}.
\end{rmk}

\begin{rmk}
The stability threshold in Theorem \ref{thm:main} matches that of the paper \cite{BVW16}. However, we expect that by using more complicated techniques developed in \cite{MasmoudiZhao22}, one can improve the stability threshold to $\ep\nu^{1/3}$. We will leave the analysis in later work.  
\end{rmk}
{\begin{rmk}
This paper focuses on the parameter regime where $\kappa\leq \nu$. The method does not seem applicable in the parameter regime $\nu\gg \kappa$. If the viscosity $\nu$ is much greater than $\kappa$, the biological phenomena take place on time scale $\mathcal{O}(\kappa^{-1})$, which can be much shorter than the fluid dynamics time scale. As a result, the dynamics of the cell evolution have a nontrivial impact on the fluid motion. Thanks to this nontrivial interaction, the fluid flow might no longer be stable. Hence it is a great problem to understand the long time dynamics of the system in this regime.  
\end{rmk}}

The remaining part of the paper is organized as follows. In Section \ref{Sect:Bootstrap}, we sketch the proof of Theorem \ref{thm:main}. In Section \ref{Sect:Fluid}, we provide estimates of the fluid equation. In Section \ref{Sect:Cell}, the estimates of the cell density are derived. Finally, we collect technical lemmas in the appendix.    
\section{Sketch of the Proof}\label{Sect:Bootstrap}
In this section, we present the main idea of the proof of  Theorem \ref{thm:main}. 

Similar to the works  \cite{Kelvin87,Orr07,Zillinger2014, BM13}, we first consider a new coordinate system
\begin{align}\label{zy_coordinate}
z=x-ty,\quad y=y. 
\end{align} we have the following:
\begin{align}\label{PKS_NS_zy_coordinate}
\pa_t N+  U\cdot \na_{L} N+\kappa \na_L\cdot(N\na_L C )=\kappa\de_L N,\quad {C=(1-\de_L)^{-1}N};\\
\pa_t \Omega+  U\cdot \na_L \Omega=\nu \de_L\Omega+\kappa\na_L^\perp\cdot (N\na_L C),\quad U=-\na_L^\perp(-\de_L)^{-1}\Omega.
\end{align}
Here the following notations are adopted:
\begin{align}
\na_L:=\left(\begin{array}{cc}\pa_z\\ \pa_y^t \end{array}\right):=\left(\begin{array}{cc}\pa_z\\ \pa_y-t\pa_z\end{array}\right), \quad \na_L^\perp:=\left(\begin{array}{cc} -\pa_y^t\\\pa_z \end{array}\right),\quad \de_L=\na_L\cdot \na_L=\pa_{zz}+(\pa_y-t\pa_z)^2.
\end{align}
Since the enhanced dissipation phenomenon is heterogeneous, we consider the $z$-average $f_0$ and the remainder $f_\nq$ of functions on $\Torus\times\rr$:
\begin{align}f_0(t,y)=\frac{1}{|\mathbb{T}|}\int_{\Torus} f(t,z,y)dz,\quad 
f_\nq(t,z,y)=f(t,z,y)-f_0(t,y).\label{x_average_remainder}
\end{align}
We decompose the solutions $N,\om$ into the {$z$}-average and remainder:
\begin{subequations}\label{PKSNS_New_Coordinate}
\begin{align}
\pa_t N_\nq+(U\cdot \na_L N)_\nq=&\kappa \de_L N_\nq-\kappa (\na_L\cdot(N\na_LC))_\nq;\label{N_nq}\\
\pa_t N_0+(U\cdot \na_L N)_0=&\kappa\pa_{yy}N_0-\kappa(\na _L\cdot (N\na_L C ))_0;\label{N_0}\\  
\pa_t \Omega_{\neq}+(U\cdot \na _L \Omega)_{\neq}=&\nu\de_L\Omega_{\neq}+\kappa (\na_L^\perp\cdot(N\na_LC))_{\neq},\label{Omega_nq}\\
\pa_t \Omega_{0}+(U\cdot \na _L \Omega)_{0}=&\nu\pa_{yy}\Omega_{0}+\kappa (\na_L^\perp\cdot(N\na_LC))_{0} .\label{Omega_0}
\end{align}
\end{subequations}

To analyze the above equations, we apply the spacial Fourier transform $f(t,z,y)\xrightarrow{\mathcal{F}} \wh f(t,k,\eta)$.  Next we introduce the following multipliers $W_{\kappa},\, W_{\nu},\,\mathcal W$, which are motivated by \cite{BVW16}:
\begin{align}
W_{\kappa}(t,k,\eta) = &\pi-\arctan\left(\kappa^{1/3}|k|^{2/3}\left(t-\frac{\eta}{k}\right)\right)\mathbf{1}_{0<|k|\leq \kappa^{-1/2}};\label{W_kappa}\\
W_{\nu}(t,k,\eta) = &\pi-\arctan\left(\nu^{1/3}|k|^{2/3}\left(t-\frac{\eta}{k}\right)\right)\mathbf{1}_{0<|k|\leq \nu^{-1/2}};\label{W_nu}\\
\mathcal W(t,k,\eta)=&\pi-\arctan\left(t-\frac{\eta}{k}\right)\mathbf{1}_{k\neq0}.\label{W}
\end{align}
We observe that these multiplier functions take values in $[\frac{\pi}{2}, \frac{3\pi}{2}]$. Moreover, they are monotonically decreasing in time. We further define  the following multipliers associated with the cell density $N$ and the vorticity $\om$:
\begin{align}
M_\kappa=W_\kappa \mathcal W,\quad M_\nu=W_\nu \mathcal W.\label{M_N_Omega}
\end{align} 
Further define that 
\begin{align}\label{A_N_Omega}
A_\iota(t,k,\eta) = &M_\iota(t,k,\eta) {e^{\delta\kappa^{1/3}|k|^{2/3}t}}(1+ |k|^2+|\eta|^2)^{s/2},\quad \iota\in\{\nu,\kappa\},\quad s\geq 0.
\end{align}
Here $0<\delta<1$ is a universal constant. We note that the $A_\nu$-multiplier and $A_\kappa$-multiplier share the same exponential factor $e^{\delta \kappa^{1/3}|k|^{2/3}t}$. 
The multipliers $\{A_\kappa, M_\kappa\}$ will act on the cell density and chemical density $N,\, C$ and $\{A_\nu, M_\nu\}$ will act on the vorticity and velocity of the fluid $\Omega,\,  U$. The properties of these multipliers are collected in the appendix. 


Next, we present a local well-posedness result, which can be proven through standard argument. 
\begin{theorem}\label{thm:local_existence}
Consider solutions $N,\om$ to the  equation \eqref{PKS_NS_zy_coordinate} subject to initial data ${0\leq} N_{\mathrm{in}}\in L^1\cap H^{s}(\Torus\times \rr)$, ${N_\mathrm{in}|y|^2\in L^1(\Torus\times\rr)}$,  $\om_{\mathrm{in}}\in H^{s}(\Torus\times\rr),\,  {3\leq s\in\mathbb N}$. There exists a small constant $T_\varepsilon(\|N_{\mathrm{in}}\|_{L^1\cap H^s},\|\om_{\mathrm{in}}\|_{H^{s}})$ such that the unique solution exists on the time interval  $[0,T_\varepsilon]$. Moreover, $N\geq 0$ on $[0, T_\varepsilon].$
\end{theorem}
\myr{Extra care is needed here because we are treating the Couette flow. We comes to the $z,y$-plane. The local a-priori energy estimate yields the $H^s_{z,y}$ bound. We should also be able to propagate the $N|y|^2$ bound because the $y\pa_x$-flow is horizontal. We use the $H^s_{z,y}$-bound (and some bounded domain approximation) to gain the compactness from Aubin-Lion. And we check that the limit is a solution. Also note that the flow is autonomous, so for any reference time $s$, we can switch it back to the \eqref{pePKS-NS}, and shift $s$ to zero and then do the  $z,y$ coordinate. Even though the existence time might now depend on the value $s$ (because the shift back will change $\pa_y$-estimate on $(z,y)$ to $\pa_y+t\pa_x$ estimate on $(x,y)$), we just need the fact that the norm can not jump in time to close the bootstrap.}
To prove Theorem \ref{thm:main}, we use a bootstrap argument. Assume that $[0,T_\star]$ is the largest time interval  on which the following \textbf{Hypotheses} hold:
\begin{subequations}\label{Hypotheses}
\begin{align}\label{Hypothesis_cell_density_enhanced_dissipation} 
\| A_\kappa N_{\neq}(t)\|_{L^2}^2+\int_0^{t}\norm{\sqrt{\frac{-\pa_\tau M_\kappa}{ M_\kappa}}A_ \kappa  N_{\neq}}_{L^2}^2+\kappa \|A_\kappa\sqrt{-\de_L} N_{\neq}\|_{L^2}^2d\tau  \leq& 2\mathcal{B}_{N_\nq}^2;\\ 
\label{Hypothesis_N_0}\|N_0(t)\|_{H^s}^2\leq &2\mathcal{B}_{N_0}^2;\\
\label{Hypothesis_omega_enhanced_dissipation} 
\| A_\nu \om_{\neq}(t)\|_{L^2}^2+\int_0^{t}\left\|\sqrt{\frac{-\pa_\tau M_\nu}{ M_\nu}}A_\nu  \om_{\neq}\right\|_{L^2}^2+\nu \|A_\nu\sqrt{-\de_L} \Omega_{\neq}\|_{L^2}^2d\tau \leq&  2\mathcal{B}_{\om_\nq}^2\ep^2\nu;\\
\label{Hypothesis_Omega_0}
\|\Omega_0(t)\|_{H^s }^2+\nu\int_0^t\|\pa_y\Omega_0\|_{H^s}^2d\tau\leq 2\mathcal{B}_{\om_0}^2 \ep^2\nu ,&\quad \forall t\in[0,T_\star].
\end{align} 
\end{subequations}
Without loss of generality, we set $\mathcal{B}_{N_\nq}, \mathcal{B}_{N_0}, \mathcal{B}_{\om_\nq}, \mathcal{B}_{\om_0}\geq 1.$ {Moreover, they only depend on the initial data $\|N_{\text{in}}\|_{L^1\cap H^s}$ and the regularity level $s$. }

\begin{pro}\label{pro_bootstrap}\myb{
Consider the system \eqref{PKS_NS_zy_coordinate} subject to  initial condition ${0\leq}N_\mathrm{in}\in L^1\cap H^s(\Torus\times\rr)$, ${N_\mathrm{in}|y|^2\in L^1(\Torus\times\rr)}$, $\Omega_{\mathrm{in}}\in H^s(\Torus\times\rr), \, { 5\leq s\in \mathbb{N}}$. Assume that $0<\kappa\leq \nu\leq 1$. Let $[0, T_\star]$ be the largest interval on which the hypotheses \eqref{Hypotheses} hold.   
There exists a threshold $\ep_0=\ep_0(\|N_{\mathrm{in}}\|_{L^1\cap H^s})$ such that if the condition \eqref{initial_data} is satisfied, the following stronger estimates hold}
\begin{subequations}\label{Conclusions}
\begin{align}
\label{Conclusion_cell_density_enhanced_dissipation} 
\| A_\kappa N_{\neq}(t)\|_{L^2}^2+\int_0^{t}\left\|\sqrt{\frac{-\pa_\tau M_\kappa }{M_\kappa}}A_\kappa N_{\neq}\right\|_{L^2}^2+\kappa \|A_\kappa\sqrt{-\de_L }N_{\neq}\|_{L^2}^2d\tau \leq& \mathcal{B}_{N_\nq}^2;\\
\label{Conclusion_N_0}
\|N_0(t)\|_{H^s}^2\leq &\mathcal{B}_{N_0}^2;\\
\label{Conclusion_omega_enhanced_dissipation} 
\|A_\nu \om_{\neq}(t)\|_{L^2}^2+\int_0^{t}\norm{\sqrt{\frac{-\pa_\tau M_\nu} {M_\nu}} A_\nu\om_{\neq}}_{L^2}^2+\nu \|A_\nu\sqrt{-\de_L} \Omega_{\neq}\|_{L^2}^2d\tau  
\leq &\mathcal{B}_{\om_\nq}^2\ep^2\nu;\\ 
\label{Conclusion_Omega_0}
\|\Omega_0(t)\|_{H^s}^2+\nu\int_0^t\|\pa_y\Omega_0\|_{H^s}^2d\tau\leq&  \mathcal{B}_{\om_0}^2\ep^2\nu,\quad \forall t\in[0,T_\star].
\end{align}
\end{subequations}
Here the bounds $\mathcal{B}_{N_\nq},\, \mathcal{B}_{N_0},\, \mathcal{B}_{\om_\nq}$ and $ \mathcal{B}_{\om_0}$ depend only on the initial data $\|N_{\mathrm{in}}\|_{L^1\cap H^s}$. 
\end{pro}
\begin{rmk}
The explicit choice of the constants $\mathcal{B}_{N_\nq},\, \mathcal{B}_{N_0},\, \mathcal{B}_{\om_\nq}$ and $ \mathcal{B}_{\om_0}$ are listed in \eqref{Choice_of_bootstrap_const}. 
The choice of the threshold $\ep_0$ can be found in \eqref{Choice_of_ep}. 
\end{rmk}
\myb{
Now we can conclude the proof of Theorem \ref{thm:main}.
\begin{proof}[Proof of Theorem \ref{thm:main}]
Combining Theorem \ref{thm:local_existence} and Proposition \ref{pro_bootstrap}, we see that $[0,T_\star]$ is both open and closed on $[0,\infty)$. Hence $T_\star=\infty$ and the solutions with estimates \eqref{Conclusions} exist for all time. The enhanced dissipation estimate \eqref{ED} is a consequence of the bounds \eqref{Conclusion_cell_density_enhanced_dissipation}, \eqref{Conclusion_omega_enhanced_dissipation}, 
\begin{align*} &\left\|e^{\delta \kappa^{1/3}|\pa_x|^{2/3}t}\left(n-\frac{1}{|\Torus|}\int_\Torus n dx\right)\right\|_{L_{x,y}^2} +\left\|e^{\delta \kappa^{1/3}|\pa_x|^{2/3}t}\left(\omega-\frac{1}{|\Torus|}\int_\Torus \omega dx\right)\right\|_{L_{x,y}^2}\\
& \quad\quad\leq
C\|A_\kappa N_\nq(t)\|_{L_{z,y}^2}+C\|A_\nu \Omega_\nq(t)\|_{L_{z,y}^2}  \leq C\mathcal{B}_{N_\nq}^2(\|N_{\mathrm{in}}\|_{L^1\cap H^s})+C\mathcal{B}_{\Omega_\nq}^2(\|N_{\mathrm{in}}\|_{L^1\cap H^s}),\quad \forall t\in[0,\infty).
\end{align*} 
This concludes the proof.
\end{proof}
}
\subsection{Notations}

Throughout the paper, the constant $C$, \myb{which can only depend on the regularity level $s$,} will change from line to line. Constants with subscript, i.e., $ \mathcal{B}_{N_\nq}$, will be fixed. 
For $A,B\geq 0$, we use the notation $A\approx B$ to highlight that there exists a strictly positive constant $C$ such that $B/C\leq A\leq CB$. We also use the notation $A\lesssim B$ to represent that there exists a constant $C>0$ such that $A\leq CB$. 

Recall the classical $L^p$ norms and Sobolev $H^s,\,s\in \mathbb{N}_+$ norms:
\begin{align}
\|f\|_{L^p}=&\|f\|_p=\left(\int |f|^p dV\right)^{1/p};\quad \|f\|_{L_t^q([0,T]; L^p )}=\left(\int_0^T\|f(t )\|_{L ^p}^qdt\right)^{1/q};\\
\|f\|_{H ^s}=&\left(\sum_{i+j \leq s}\|\pa_z^i\pa_y^j  f\|_{L^2}^2\right)^{1/2};\quad
\|f\|_{\dot H^s}=\left(\sum_{i+j =s}\|\pa_z^i\pa_y^j  f\|_{L ^2}^2\right)^{1/2}. 
\end{align} 
Here $dV=dzdy=dxdy$ is the volume element. If $p$ or $q$ is $\infty$, then we use the classical $L^\infty$-norm. 

We use $\widehat{f}$ to denote Fourier transform of function $f$ in the  $(z,y)$ variables. 
The frequency variables corresponding $z$ and $y$ are denoted by $\{k,\ell\}$ and $\{\eta,\xi\}$, respectively.  
\ifx
The Fourier transform and its inverse has the following form
\begin{equation*}
\wh{f}(k,\eta) := \frac{1}{(2\pi)^2}\int_0^{2\pi}\int_{\rr} e^{-ik z-i\eta y} f(z,y) dzdy , \quad \check{g}(z,y) = \sum_{\al=-\infty}^\infty \int g(k,\eta) e^{ik z+iy\eta}d\eta.
\end{equation*}
\fi
The Fourier multipliers are defined as follows:
\begin{align}
\widehat{(\mathfrak M(\na ) f)}(k,\eta)=\mathfrak M(ik,i\eta)\wh f(k,\eta).
\end{align}
Given a set $\mathcal{S}\in \mathbb{Z}$, we define the projection $\mathbb{P}_{k\in \mathcal{S}}$
\begin{align}
\widehat{\mathbb{P}_{k\in \mathcal{S}} f}=\mathbf{1}_{k\in\mathcal{S}} \wh f(k,\eta).
\end{align} 
We apply the following notations
\begin{align}
\lan k\ran= \sqrt{1+k^2},\,\quad \lan k,\eta\ran =(1+k^2+\eta^2)^{1/2},\quad |k,\eta|=|k|+|\eta|.
\end{align} We recall from classical literature that $\|f\|_{H^s}\approx \|\lan \pa_z,\pa_y \ran^s f\|_{L^2}.$
\section{Vorticity Estimates}\label{Sect:Fluid}
In this section, we derive the estimates associated with the fluid motion. We organize the proof into two subsections. In Subsection \ref{Sect:Fluid_1}, we prove the enhanced dissipation estimate of the vorticity remainder $\Omega_\nq$ \eqref{Conclusion_omega_enhanced_dissipation}. In Subsection \ref{Sect:Fluid_2}, we prove the $z$-average ($\Omega_0$) estimate \eqref{Conclusion_Omega_0}.
\subsection{Remainder Estimates}\label{Sect:Fluid_1}
In this section, we consider the remainder of the vorticity, which solves the equation \eqref{Omega_nq}. We will prove the conclusion \eqref{Conclusion_omega_enhanced_dissipation}.  

Application of the energy estimate yields that
\begin{align}
\frac{d}{dt}\frac{1}{2}\|A_\nu\Omega_{\neq}\|_{L^2}^2=&\delta \kappa^{1/3}\sum_{k\neq 0}|k|^{2/3}\int |A_\nu(t,k,\eta) \wh \Omega(t,k,\eta)|^2d\eta -\sum_{k\neq0}\int \frac{-\pa_t M_\nu}{ M_\nu}   |A_\nu(t,k,\eta)\wh \om(t,k,\eta)|^2 d\eta\\
&-\nu \sum_{k\neq 0}\int |A_\nu(t,k,\eta)(\sqrt{|k|^2+|\eta-kt|^2})\wh \Omega (t,k,\eta)|^2d\eta\\
&-\int A_\nu(U\cdot \na _L \Omega )_\nq \ A_\nu\Omega_{\neq} dV+\kappa \int  A_\nu(\na_L^\perp\cdot (N \na_L C))_{\neq} \ A_\nu\Omega_{\neq}dV.
\end{align}
The relation $\kappa\leq \nu$, together with the property of the multiplier $M_\nu$ \eqref{M_property_ED} yields that if $\delta \leq \frac{1}{16\pi^2}$, then the time evolution of the norm is bounded as follows:
\begin{align}
\frac{d}{dt}\frac{1}{2}\|A_\nu\Omega_{\neq}\|_{L^2}^2\leq &-\frac{1}{2}\left\|\sqrt{\frac{-\pa_t M_\nu}{M_\nu}}A_\nu \om_\nq\right\|_{L^2}^2-\frac{1}{2}\nu\|A_\nu\sqrt{-\de_L}\om_\nq\|_{L^2}^2\\
&-\int A_\nu(U\cdot \na _L \Omega )_\nq  \ A_\nu\Omega_{\neq} dV+\kappa \int A_\nu(\na_L^\perp\cdot (N \na_L C))_{\neq}  \ A_\nu\Omega_{\neq}dV.
\end{align} 
Integration in time yields that
\begin{align}
\|A_\nu\Omega_\nq(t)\|_{L^2}^2&+\int_0^t \left\|\sqrt{\frac{-\pa_\tau M_\nu}{M_\nu}}A_\nu \Omega_\nq\right\|_{L^2}^2d\tau+\nu\int_0^t\|A_\nu \sqrt{-\de_L}\om_\nq\|_{L^2}^2d\tau\\
\leq& \|A_\nu\Omega_{\text{in};\nq}\|_{L^2}^2+2\bigg| \int _0^t\int A_\nu(U\cdot \na _L \Omega )_\nq \  A_\nu\Omega_{\neq} dVd\tau\bigg|+2\bigg|\kappa \int_0^t\int A_\nu(\na_L^\perp\cdot (N \na_L C))_{\neq}\ A_\nu\Omega_{\neq}  dV d\tau\bigg|\\
=:&\|A_\nu\Omega_{\text{in};\nq}\|_{L^2}^2+T_{\Omega_\nq;1}+T_{\Omega_\nq;2}.\label{T12}
\end{align}
The estimates of the terms $T_{\Omega_\nq;1}$ and $T_{\Omega_\nq;2}$ are summarized in Lemma \ref{lem:T_1} and Lemma \ref{lem:T_2}.
\begin{lem}\label{lem:T_1}
Assume all the conditions in Proposition \ref{pro_bootstrap}. The  $T_{\Omega_\nq;1}$-term in \eqref{T12} is bounded as follows
\begin{align}\label{T_1_bound}
T_{\Omega_\nq;1}\leq C(\mathcal{B}_{\om_\nq}+\mathcal{B}_{\om_0} )\mathcal{B}_{\om_\nq}^2\ep^3\nu .
\end{align}
Here the constant $C$ depends only on the regularity level $s$. 
\end{lem}

\begin{lem}\label{lem:T_2}
Assume all the conditions in Proposition \ref{pro_bootstrap}. There exists a constant $C=C(s)$ such that the following estimate of the $T_{\om_\nq;2}$-term holds 
\begin{align}\label{T_2_bound}
T_{\om_\nq;2}\leq \frac
{1}{4}\nu\int_0^t\|A_\nu \sqrt{-\de_L} \om_\nq\|_{L^2}^2d\tau+\frac{1}{4}\int_0^t\left\|A_\nu\sqrt{\frac{-\pa_\tau M_\nu}{M_\nu}}\om_\nq\right\|_{L^2}^2d\tau+ 	C (\mathcal{B}_{N_\nq}^4+\mathcal{B}_{N_0}^4 )\ep^2 \nu.
\end{align} 
\end{lem}
With these two lemmas, we are ready to prove the improved bound \eqref{Conclusion_omega_enhanced_dissipation}.
\begin{proof}[Proof of the conclusion \eqref{Conclusion_omega_enhanced_dissipation}]
First, we recall the initial condition \eqref{initial_data}, which ensures that $\|A_\nu \om_{\text{in};\nq}\|_{L_{z,y}^2}\leq C \|\om_{\text{in};\nq}\|_{H_{x,y}^s}\leq C\ep\nu^{1/2}$. Combining it with the relation \eqref{T12} and the estimates \eqref{T_1_bound}, \eqref{T_2_bound}, we obtain that
\begin{align}
&\|A_\nu \om_\nq(t)\|_{L^2}^2+\int_0^t\left\|\sqrt{\frac{-\pa_\tau M_\nu }{M_\nu}}\Omega_\nq(\tau)\right\|_{L^2}^2d\tau+\nu\int_0^t\|A_\nu \sqrt{-\de_L}\om_\nq(\tau)\|_{L^2}^2d\tau\\
&\leq C\|\Omega_{\text{in};\nq}\|_{H^s}^2+C\left(\ep(\mathcal{B}_{\om_\nq}+\mathcal{B}_{\om_0})\mathcal{B}_{\om_\nq}^2+\mathcal{B}_{N_\nq}^4+\mathcal{B}_{N_0}^4\right)\ep^2\nu\leq C\left(\ep(\mathcal{B}_{\om_\nq}+\mathcal{B}_{\om_0})\mathcal{B}_{\om_\nq}^2+1+\mathcal{B}_{N_\nq}^4+\mathcal{B}_{N_0}^4\right)\ep^2\nu .
\end{align}
Here the constant $C$ depends only on the regularity level $s$. 
As a result the following choices of constants yields  \eqref{Conclusion_omega_enhanced_dissipation}\begin{align}\label{Choice_of_Const_1} 
\frac{1}{2}\mathcal{B}_{\om_\nq}^2\geq C(1+\mathcal{B}_{N_\nq}^4+\mathcal{B}_{N_0}^4), \quad \ep\leq \frac{1}{2C(\mathcal{B}_{\om_\nq}+\mathcal{B}_{\om_0})}. 
\end{align}
This concludes the proof.
\end{proof}

In the remaining part of this subsection, we prove Lemma \ref{lem:T_1} and Lemma \ref{lem:T_2}. 
\begin{proof}[Proof of Lemma \ref{lem:T_1}]The estimate of the $T_{\Omega_\nq;1}$-term is in the same vein as the one in \cite{BVW16}. We carry out the details for the sake of completeness. We recall that the Biot-Savart law yields that $U_0^{(2)}=\pa_z\de_y^{-1}\Omega_0=0$. Hence we can  expand the $T_{\Omega_\nq;1}$ term as follows:
\begin{align}
T_{\Omega_\nq;1}\leq&2\bigg|\int _0^t\int A_\nu(U_\nq\cdot \na _L \Omega )\  A_\nu\Omega_{\neq} dVd\tau\bigg|+2\bigg|\int_0^t\int A_\nu (U_0^{(1)} \pa_z \om)\ A_\nu \om_\nq dVd\tau\bigg|=:\sum_{i=1}^2T_{\Omega_\nq;1i}.\label{T_1_12}
\end{align}
For the  $T_{\Omega_\nq;11}$-term, we first invoke the Biot-Savart law to rewrite the velocity as $U_\nq=\na_L^\perp \de_L^{-1}\om_\nq$, and then apply the product estimate associated with the multiplier $A_\nu$  \eqref{A_product_rule_Hs} to obtain the  following bound
\begin{align}
T_{\Omega_\nq;11}\leq& C\int_0^t \|A_{\nu}(\na_L^\perp\de_L^{-1}\om_\nq \cdot \na_L \om )\|_{L^2}\|A_\nu \om_\nq\|_{L^2}d\tau 
\leq C\int _0^t\|A_\nu \na^\perp_L \de_L^{-1}\Omega_\nq\|_{L^2}\|A_\nu\na_L\om\|_{L^2}d\tau\|A_\nu\Omega_\nq\|_{L_t^\infty L^2}.
\end{align} Next we observe that the $M_\nu$-properties \eqref{M_bound}, \eqref{M_property_common_dot_M} imply that $\|A_\nu{\na_L^\perp} \de_L^{-1}\om_\nq\|_{L^2}\leq C\|A_\nu \sqrt{\frac{-\pa_t M_\nu}{M_\nu}}\om_\nq\|_{L^2}$. Combining this with the bootstrap hypothesis \eqref{Hypothesis_omega_enhanced_dissipation}, \eqref{Hypothesis_Omega_0} yields that
\begin{align}
T_{\Omega_\nq;11}\leq &\frac{C}{\sqrt\nu}\norm{A_\nu\sqrt{\frac{-\pa_t M_\nu}{M_\nu}}\om_\nq}_{L_t^2L^2}(\nu^{1/2}\|A_\nu\sqrt{-\de_L }\om_\nq\|_{L_t^2L^2}+\nu^{1/2}\|\pa_y \om_0\|_{L_t^2 H^s})\|A_\nu\Omega_\nq\|_{L_t^\infty L^2}\\
\leq &C\ep(\mathcal{B}_{\om_\nq }+\mathcal{B}_{\om_0})\mathcal{B}_{\om_\nq}^2\ep^2 \nu. \label{T_om_nq_11}
\end{align}
This is consistent with the estimate \eqref{T_1_bound}. 

For the $T_{\Omega_\nq;12}$ term in \eqref{T_1_12}, we first observe there is a cancellation relation 
\begin{align}
\int \pa_y\de_y^{-1}\om_0 \ \pa_z A_\nu\om_\nq  \ A_\nu \om_\nq dV=\int \pa_y\de_y^{-1}\om_0 \pa_z\left(\frac{A_\nu \om_\nq}{2}\right)^2dV=0.
\end{align}
Hence we can rewrite the term $T_{\Omega_\nq;12}$ using the Biot-Savart law $U_0^{(1)}=-\pa_y\de_y^{-1}\om_0$ as follows:
\begin{align}
T&_{\Omega_\nq;12} =2\bigg|\int_0^t \int \left( A_\nu(-\pa_y\de_y^{-1}\om_0\  \pa_z \om_\nq)+\pa_y\de_y^{-1} \om_0\ \pa_z A_\nu \om_\nq \right) A_\nu \om_\nq\  dVd\tau\bigg|\\
=&C\bigg|\sum_{k\neq0}\int_0^t\iint \left(M_\nu(\tau,k,\eta)\lan k,\eta\ran^s-M_\nu (\tau,k,\xi)\lan k,\xi\ran^s\right)\ \frac{i(\eta-\xi)}{(\eta-\xi)^2}\wh \om_0(\eta-\xi)\\
&\quad\quad\quad\quad \quad\quad\times \left( k e^{\delta\kappa^{1/3}|k|^{2/3}\tau}\wh  \om(\tau,k,\xi)\right)\ \overline{{A_\nu \wh\om(\tau,k,\eta)}} \ d\eta d\xi d\tau\bigg|.\label{commutator_form}
\end{align}\myb{
Now we invoke the commutator estimate \eqref{commutator_estimate} and the Young's convolution inequality to obtain that
\begin{align}
&T_{\Omega_\nq;12}\\
 &\leq C\bigg|\int_0^t\sum_{k\neq0}\iint \left(\lan k,\xi\ran^s+\lan\eta-\xi\ran^s \right)\ |\wh \om_0(\eta-\xi)|\left|  e^{\delta\kappa^{1/3}|k|^{2/3}\tau}\wh  \om(\tau,k,\xi)\right|\ | {A_\nu \wh\om(\tau,k,\eta)}| \ d\eta d\xi d\tau\bigg|\\
&\leq C\bigg|\int_0^t \sum_{k\nq 0}\|A_\nu\wh  \om_k(\cdot)\|_{L_\eta^2}\left(\|\lan \cdot\ran^s \wh \om_0(\cdot)\|_{L_\eta^2}\|e^{\delta\kappa^{1/3}|k|^{2/3}\tau}\wh \om_k( \cdot)\|_{L_\eta^1}+\|\wh\om_0(\cdot)\|_{L_\eta^1}\|e^{\delta \kappa^{1/3}|k|^{2/3}\tau}\lan k,\cdot\ran^s\wh \om_k(\cdot)\|_{L^2_\eta}\right)\ d\tau\bigg|.
\end{align}Combining the $M_\nu$-properties \eqref{M_bound},  \eqref{M_property_ED}, the definition of $A_\nu$ \eqref{A_N_Omega} and the inequality $\|\wh f(\cdot)\|_{L^1_\eta}\leq C\|\lan\cdot\ran^s\wh f(\cdot)\|_{L^2_\eta},\, s\geq 1$ yields that
\begin{align}
T_{\Omega_\nq;12}\leq &C\int_0^t \sum_{k\nq0}\|A_\nu \wh\om_k(\cdot)\|_{L^2_\eta}^2d\tau \|\om_0\|_{L^\infty_t H^s}\\ 
 \leq &C\nu^{-1/3}\left(\int_0^t\left\|A_\nu\sqrt{\frac{-\pa_\tau M_\nu}{M_\nu}}\om_\nq\right\|_{L^2}^2 +\nu\|A_\nu\sqrt{-\de_L}\Omega_\nq\|_{L^2}^2 d\tau\right)\|\om_0\|_{L^\infty_t H^s} .
 \end{align}Hence the bootstrap hypotheses \eqref{Hypothesis_omega_enhanced_dissipation}, \eqref{Hypothesis_Omega_0} implies that
\begin{align}
T_{\Omega_\nq;12} \leq  C\ep^3 \nu\mathcal{B}_{\om_\nq}^2  \mathcal{B}_{\om_0}.\label{T_om_nq_12}
\end{align}}
Combining the decomposition \eqref{T_1_12} and the estimates \eqref{T_om_nq_11},  \eqref{T_om_nq_12}, we have obtained the result \eqref{T_1_bound}.
\end{proof}

\begin{proof}[Proof of Lemma \ref{lem:T_2}]
We further decompose the $T_{\Omega_\nq;2}$ term in \eqref{T12} as follows:\begin{align}
T_{\Omega_\nq;2}\leq &2\kappa \bigg|\int_0^{t}\int  A_\nu(\pa_z  (N_{\neq} \pa_y  C_0))\ A_\nu \Omega_{\neq}dVd\tau\bigg|+2\kappa\bigg| \int_0^{t}\int A_\nu(\na_L^\perp\cdot (N \na_L C_{\neq}))\  A_\nu\Omega_{\neq} dVd\tau\bigg|\\
=:&T_{\Omega_\nq;21}+T_{\Omega_\nq;22}.\label{T_om_nq_2_12}
\end{align} 
Before analyzing the terms, we make a comment about the multipliers. Thanks to the $\{M_\nu,M_\kappa\}$-property \eqref{M_bound} and the definitions of $\{A_\nu,A_\kappa\}$ \eqref{A_N_Omega}, we have that the $A_\nu$, $A_\kappa$ multipliers are comparable, i.e., 
\begin{align}\label{A_nu_kappa_comparison}
\frac{1}{16\pi^4}A_\kappa (t,k,\eta)\leq A_\nu(t,k,\eta)\leq 16\pi ^4 A_\kappa(t,k,\eta). 
\end{align}
As a result, we have the freedom to adjust the multipliers $\{A_\iota\}_{\iota\in\{\kappa,\nu\}}$ when considering different objects. 

With the multiplier properties explained, we start the estimate.
\ifx by decomposing the $T_{\Omega_\nq;21}$ term as follows:
\begin{align}
T_{\Omega_\nq;21}\leq&\bigg|C\kappa \sum_{|k|\in(0,\kappa^{-1/2}] }\int_0^t \iint A_\nu(\tau,k,\eta) \left(ik \wh N(\tau,k,\xi)  {i(\eta-\xi)} \wh C_0(\tau,\eta-\xi) \right) \ \overline{A_\nu \wh\Omega(\tau,k,\eta)}d\eta d\xi d\tau \bigg|\\
&+   \bigg|C\kappa \sum_{|k|>\kappa^{-1/2} }\int_0^t \iint A_\nu(\tau,k,\eta) \left(ik \wh N(\tau,k,\xi)  {i(\eta-\xi)} \wh C_0(\tau,\eta-\xi) \right)\ \overline{A_\nu  \wh\Omega(\tau,k,\eta)} d\eta d\xi d\tau\bigg|\\
=:&T_{\Omega_\nq;21}^\text{low} +T_{\Omega_\nq;21}^\text{high}.\label{T_om_nq_21_l_h}
\end{align}
\fi
To estimate the $T_{\Omega_\nq;21}$-term, we apply the relation $(1-\pa_{yy})C_0=N_0$, the definition \eqref{A_N_Omega}, and the $\pa_t M_\kappa$-estmate \eqref{M_property_common_dot_M} to get the following \begin{align}
T_{\Omega_\nq;21}
\leq &C\kappa \sum_{k\nq 0 }\int_0^t\iint\bigg| \sqrt{|k|^2+|\eta-k\tau|^2}{A_\nu  \wh\Omega(\tau,k,\eta)}\bigg| \\
&  \quad\quad\quad\quad \quad\quad\quad\times\bigg| A_\kappa(\tau,k,\eta)  \sqrt{\frac{|k|^2}{|k|^2+|\eta-k\tau|^2}}\left( | \wh N(\tau,k,\xi)  |\frac{|\eta-\xi|}{1+|\eta-\xi|^2}| \wh N_0(\tau,\eta-\xi)| \right)\bigg|d\eta d\xi d\tau\\
\leq &C\kappa \sum_{k\nq 0}\int_0^t\iint\bigg|\sqrt{|k|^2+|\eta-k\tau|^2}{ A_\nu\wh\Omega(\tau,k,\eta)}\bigg| \\
&  \quad\times \bigg|M_\kappa(\tau,k,\eta)\lan k,\eta \ran ^{s}  \sqrt{ -\pa_\tau M_\kappa(\tau,k,\eta) }\left( e^{\delta\kappa^{1/3}|k|^{2/3}\tau} |\wh N(\tau,k,\xi)|  \frac{|\eta-\xi|}{1+|\eta-\xi|^2}| \wh N_0(\tau,\eta-\xi)| \right)\bigg|d\eta d\xi d\tau.
\end{align}
Now we invoke the $M_\kappa$ properties \eqref{M_bound},  \eqref{M_5} to obtain
\begin{align}
&T_{\Omega_\nq;21}\\
 &\leq C\kappa \sum_{k\nq 0}\int_0^t\iint\bigg|\sqrt{|k|^2+|\eta-k\tau|^2} { A_\nu\wh\Omega(\tau,k,\eta)}\bigg| \\
&\quad\times(M_\kappa(\tau,k,\xi)\lan k,\xi \ran ^{s}+\lan \eta-\xi\ran^s) e^{\delta\kappa^{1/3}|k|^{2/3}\tau} \sqrt{ -\pa_\tau M_\kappa(\tau,k,\xi) }|\wh N(\tau,k,\xi)| \ \frac{|\eta-\xi|\lan \eta-\xi\ran}{1+|\eta-\xi|^2} |\wh N_0(\tau,\eta-\xi) | d\eta d\xi d\tau. 
\end{align}\myb{
Now we apply similar argument as in \eqref{T_om_nq_12} to estimate the term. Application of Young's convolution inequality, the $M_\kappa$-bound \eqref{M_bound}, the $A_\kappa$-definition \eqref{A_N_Omega}, and the fact that $\|\wh f(\cdot)\|_{L^1_\eta}\leq C\|\lan\cdot\ran^s\wh f(\cdot)\|_{L^2_\eta},\, s\geq 1$  yields that
\begin{align}
T_{\Omega_\nq;21}\leq &\frac{1}{8}\nu\int_0^t\left\|A_\nu\sqrt{-\de_L}\om_\nq \right\|_{L^2}^2d\tau+C\ep^2 \nu\int_0^t\left\|A_\kappa \sqrt{\frac{-\pa_\tau M_\kappa}{M_\kappa}}N_\nq\right\|_{L^2}^2\|N_0\|_{H^s}^2d\tau\\
\leq &\frac{1}{8}\nu\int_0^t\left\|A_\nu\sqrt{-\de_L}\om_\nq \right\|_{L^2}^2d\tau+C\ep^2 \nu \mathcal{B}_{N_\nq}^2\mathcal{B}_{N_0}^2.\label{T_om_nq_2_1}
\end{align}
Here in the last line, the hypotheses \eqref{Hypothesis_cell_density_enhanced_dissipation} and \eqref{Hypothesis_N_0} are employed. }
\ifx
Next we estimate the $T_{\Omega_\nq;21}^{\text{high}}$ using the product estimate \eqref{A_product_rule_Hs} and the relation $(1-\pa_{yy})C_0=N_0$ as follows
\begin{align}
T_{\Omega_\nq;21}^{\text{high}}\leq & C\left|\kappa \sum_{|k|>\kappa^{-1/2}}\int_0^t\iint \overline{A_\nu ik\wh\Omega(\tau,k,\eta)} \ A_\nu(\tau,k,\eta) \left(\frac{k}{k} \wh N(\tau,k,\xi)  \frac{i(\eta-\xi)}{1+|\eta-\xi|^2} \wh N_0(\tau,\eta-\xi) \right)d\eta d\xi\right| \\
\leq &\frac{1}{8}\nu\int_0^t\|A_\nu \pa_z \om_\nq\|_2^2d\tau+C\ep^2 \nu\kappa \int_0^t\|A_\kappa \pa_z N_\nq\|_2^2\|  N_0\|_{H^{s}}^2 d\tau.
\end{align}
Now we invoke the hypothesis \eqref{Hypothesis_cell_density_enhanced_dissipation} to obtain that 
\begin{align}T_{\Omega_\nq;22}^\text{high}\leq &\frac{1}{8}\nu\int_0^t\|A_\nu \pa_z \om_\nq\|_2^2d\tau+C\ep^2 \nu \mathcal{B}_{N_\nq}^2\mathcal{B}_{\om_0}^2.
\end{align} This is consistent with \eqref{T_2_bound}. 
\fi
  
For the $T_{\Omega_\nq;22}$ term in \eqref{T_om_nq_2_12}, we apply integration by parts, and then estimate it with the product estimate  \eqref{A_product_rule_Hs}, the elliptic estimate \eqref{Green_M_t_est}, and the bootstrap hypotheses \eqref{Hypothesis_cell_density_enhanced_dissipation}, \eqref{Hypothesis_N_0} as follows
\myb{\begin{align}
T_{\Omega_\nq;22}\leq & \kappa\int_0^{t}\|A_\nu\na_L^\perp \om_{\neq}\|_{L^2} (\| N_{0}\|_{H^s}+\|A_\kappa N_\nq\|_{L^2})\|A _\kappa\na_L C_{\neq}\|_{L^2} d\tau\\
\leq&\frac{1}{16}\nu\int_0^{t}\|A_\nu\sqrt{-\de_L} \om_{\neq}\|_{L^2}^2d\tau+C\ep^2 \nu(\| N_0\|_{L^\infty_t H^s}^2+\|A_\kappa N_\nq\|_{L_t^\infty L^2}^2)\|A_\kappa\na_L(1-\de_L)^{-1} N_{\neq}\|_{L_t^2L^2}^2\\
\leq &\frac{1}{16}\nu\int_0^{t}\|A_\nu\sqrt{-\de_L} \om_{\neq}\|_{L^2}^2d\tau+C\ep^2 \nu(\| N_0\|_{L^\infty_t H^s}^2+\|A_\kappa N_\nq\|_{L_t^\infty L^2}^2)\left\|A_\kappa\sqrt{\frac{-\pa_t M_\kappa }{M_\kappa}}N_{\neq}\right\|_{L_t^2L^2}^2\\
\leq&\frac{1}{16}\nu\int_0^{t}\|A_\nu\sqrt{-\de_L} \om_{\neq}\|_{L^2}^2d\tau+C\ep^2\nu (\mathcal{B}_{N_0}^2+\mathcal{B}_{N_\nq}^2)\mathcal{B}_{N_\nq}^2.
\end{align}}Combining this with the decomposition \eqref{T_om_nq_2_12} and the bound \eqref{T_om_nq_2_1} yields the estimate  \eqref{T_2_bound}.

 \end{proof}

\subsection{The $z$-average Estimates}\label{Sect:Fluid_2}In this subsection, we prove the estimate \eqref{Conclusion_Omega_0}. First of all, we recall the equation \eqref{Omega_0}, and decompose the nonlinearity as follows  
\begin{align}
\pa_t \Omega_0+(U_\nq\cdot \na_L \Omega_\nq)_0+(U_0\cdot \na\Omega_0)_0=\nu \de \Omega_0+\kappa\left(\na_L^\perp\cdot(N_0\na_L C_0)\right)_0+\kappa\left(\na_L^\perp\cdot(N_\nq\na_L C_\nq)\right)_0.
\end{align}
Here we observe two null-structures which lead to simplification. First, we observe that by the Biot-Savart law, the vertical velocity field $U_0^{(2)}=\pa_z\de_y^{-1}\om_0=0$. Hence,
\begin{align}
(U_0\cdot \na \om_0)_0=(U_0^{(1)}\pa_z\om_0)_0=0.
\end{align} 
On the other hand, the following term vanishes,
\begin{align}
\kappa\left(\na_L^\perp\cdot ( N_0 \na_L C_0)\right)_0=-\kappa\left((\pa_y-t\pa_z)(N_0\pa_z C_0)\right)_0+\kappa\left(\pa_z (N_0\pa_y C_0)\right)_0=0.
\end{align}
Hence the $\om_0$-equation can be simplified to the following,
\begin{align}
\pa_t \Omega_0-\nu \de \Omega_0 =-(U_\nq\cdot \na_L \Omega_\nq)_0+\kappa\left(\na_L^\perp\cdot(N_\nq\na_L C_\nq)\right)_0.
\end{align}
\ifx 
As a result, we might have that 
\begin{align}
\|\Omega_0\|_{H^N}\leq \mathcal{B}_{\om_0}\ep.
\end{align} 
\fi 
Recalling that $A_\nu(t,k=0,\eta)={\pi^2 \lan \eta\ran^{s}}$, we calculate the time evolution of the $\|A_\nu\om_0\|_{L^2}^2$ as follows:
\begin{align}
\frac{d}{dt}\frac{1}{2}\|A_\nu\om_0\|_{L^2}^2=-\nu\|A_\nu \pa_y \om_0\|_{L^2}^2-\int A_\nu((U_\nq\cdot \na_L \Omega_\nq)_0) \ A_\nu \om_0 dV+\kappa\int  A_\nu\left(\na_L^\perp\cdot(N_\nq\na_L C_\nq)\right)_0\ A_\nu\om_0 dV.
\end{align}
Now integration in time yields
\begin{align}\label{T_om_0_12}
\|A_\nu&\om_0(t)\|_{L^2}^2+2\nu\int_0^t\|A_\nu \pa_y\om_0\|_{L^2}^2d\tau\\
 \leq &\|A_\nu\om_{\text{in};0}\|_{L^2}^2+2\bigg|\int_0^t\int A_\nu((U_\nq\cdot \na_L \Omega_\nq)_0) \ A_\nu \om_0 dVd\tau\bigg|+2\kappa\bigg|\int_0^t\int  A_\nu\left(\na_L^\perp\cdot(N_\nq\na_L C_\nq)\right)_0 \ A_\nu\om_0 dVd\tau\bigg|\\
=:&\|A_\nu\om_{\text{in};0}\|_{L^2}^2+T_{\om_0;1}+T_{\om_0;2}.
\end{align}
We rewrite the $T_{\om_0;1}$-term in \eqref{T_om_0_12} with the Biot-Savart law, and then estimate it with the product estimate \eqref{A_product_rule_Hs},  the $M_\nu$-bound \eqref{M_bound}, the elliptic estimate \eqref{Green_M_t_est}, and the  hypotheses \eqref{Hypothesis_omega_enhanced_dissipation}, \eqref{Hypothesis_Omega_0} as follows
\begin{align}
T_{\om_0;1}=& \int_0^t\int A_\nu (\na_L^\perp\de_L^{-1}\om_{\neq}\cdot \na _L \Omega_{\neq})_0\ A_\nu \Omega_0 dV  
\leq  C\int_0^{t} \|A_\nu\na_L^\perp \de_L^{-1}\om_{\neq}\|_{L^2} \|A_\nu \na _L \Omega_{\neq} \|_{L^2}\| A_\nu\Omega_0\|_{L^2}d\tau\\
\leq &C\left\|A_\nu\sqrt{\frac{-\pa_t M_\nu}{M_\nu}} \om_{\neq}\right\|_{L_t^2L^2}\|A_\nu \na_L \Omega_{\neq}\|_{L_t^2L ^2}\|\Omega_0\|_{L_t^\infty H _y^s} 
 \leq   {\ep^{2}}{\nu}\left(\ep C \mathcal{B}_{\om_0}\mathcal{B}_{\om_\nq}^2\right).\label{T_1'} 
\end{align}
Next we estimate $T_{\om_0;2}$-term in \eqref{T_om_0_12}. By observing $(\pa_z F)_0\equiv 0$ and integration by parts, we rewrite the term as follows,  
\begin{align}
T_{\om_0;2}
=&2\bigg|\kappa \int_0^t\int A_\nu\Omega_0 A_\nu((\pa_y-t\pa_z) (N_{\neq}\pa_z C_{\neq}))_0dVd\tau\bigg| 
= 2\bigg|\kappa\int_0^t \int A_\nu\pa_y\Omega_0 \ A_\nu(N_{\neq}\pa_z C_{\neq})_0dV d\tau\bigg|.
\end{align}Application of the product estimate \eqref{A_product_rule_Hs} and the fact that $A_\nu\approx A_\kappa$ \eqref{A_nu_kappa_comparison} yields that \begin{align}
T_{\om_0;2}\leq& \frac{1}{2}\nu\|A_\nu\pa_y \Omega_0\|_{L_t^2L^2}^2+C\ep^2 \nu\|A_\kappa N_{\neq}\|_{L_t^\infty L^2}^2\|A_\kappa\pa_z C_{\neq}\|_{L_t^2L^2}^2.
\end{align}
After invoking the relation $\pa_z C_\nq=\pa_z(1-\de_L)^{-1}N_\nq$, the $M_\kappa$-estimate \eqref{M_bound} and the elliptic estimate   \eqref{Green_M_t_est}, we estimate the $T_{\om_0;2}$-term with the hypotheses \eqref{Conclusion_cell_density_enhanced_dissipation} as follows
\begin{align}
T_{\om_0;2}\leq&\frac{1}{2}\nu\|A_\nu\pa_y \Omega_0\|_{L_t^2L^2}^2 +C\ep^2 \nu \|A_\kappa N_{\neq}\|^2_{L_t^\infty L^2}\|A_\kappa\pa_z(1-\de_L)^{-1} N_{\neq}\|_{L_t^2L^2}^2\\
\leq&\frac{1}{2}\nu\|A_\nu\pa_y \Omega_0\|_{L_t^2L^2}^2+ C\ep^2 \nu\|A_\kappa N_\nq\|_{L_t^\infty L^2}^2\left\|A_\kappa\sqrt{\frac{-\pa_t M_\kappa}{ M_\kappa}}N_{\neq}\right\|_{L_t^2L^2}^2 
\leq  \frac{1}{2}\nu\|A_\nu\pa_y \Omega_0\|_{L_t^2L^2}^2 +C \mathcal{B}_{N_\nq}^4\ep^2 \nu.\label{T_2'}
\end{align} 
Combining the decomposition \eqref{T_om_0_12}, the estimates \eqref{T_1'}, \eqref{T_2'}, and the initial constraint \eqref{initial_data}, we have that 
\begin{align}\label{A_nu_om_0}\quad\quad 
\|A_\nu \om_0(t)\|_{L^2}^2+\nu \|A_\nu\pa_y\om_0\|_{L_t^2L^2}^2 \leq& C\|\om_{\text{in};0}\|_{H^s}^2+\ep^2 \nu C(\ep  \mathcal{B}_{\om_0}\mathcal{B}_{\om_\nq}^2+\mathcal{B}_{N_\nq}^4) 
\leq \ep^2 \nu C(1+\ep  \mathcal{B}_{\om_0}\mathcal{B}_{\om_\nq}^2+\mathcal{B}_{N_\nq}^4).
\end{align}
Here $C\geq1 $ is a constant depending only on the regularity level $s$.  
Hence the following choice of constants guarantees the conclusion \eqref{Conclusion_Omega_0}
\begin{align}\label{Choice_of_Const_2}
\mathcal{B}_{\om_0}^2\geq 4C(1+\mathcal{B}_{\om_\nq}^2+\mathcal{B}_{N_\nq}^4),\quad \ep \leq \frac{1}{ 1+\mathcal{B}_{\om_0} }.
\end{align}
Here the constant $C$ is the one in \eqref{A_nu_om_0}. 
\section{Cell Density Estimates}\label{Sect:Cell}
In this section, we derive the estimates associated with the cell dynamics. We organize the proof into two subsections. In Subsection \ref{Sect:Cell_1}, we prove the enhanced dissipation estimate of the cell density's remainder $N_\nq$ \eqref{Conclusion_cell_density_enhanced_dissipation}. In Subsection \ref{Sect:Cell_2}, we prove the $z$-average ($N_0$) estimate \eqref{Conclusion_N_0}.
\subsection{The Remainder of the Cell Density}\label{Sect:Cell_1}
In this section, we prove the estimate \eqref{Conclusion_cell_density_enhanced_dissipation}. 
First we calculate the time evolution of $\|A_\kappa N_{\neq}\|_{L^2}^2$ using the equation  \eqref{N_nq}: 
\begin{align}
\frac{1}{2}\|A_\kappa N_{\neq}(t)\|_{L^2}^2 
=&\frac{1}{2}\|A_\kappa N_{\text{in};\neq}\|_{L^2}^2+\delta\kappa^{1/3}\int_0^t\||\pa_x|^{1/3}A_\kappa N_{\neq}\|_{L^2}^2d\tau-\int_0^t \left\|\sqrt{\frac{-\pa_tM_\kappa}{  M_\kappa}}A_\kappa N_{\neq}\right\|_{L^2}^2d\tau\\
&-\kappa \int_0^t\|A_\kappa\sqrt{-\de_L}N_{\neq}\|_{ L^2}^2d\tau-\int_0^t\int A_\kappa N_\nq \ A_\kappa (U\cdot \na_L N)_\nq dVd\tau\nb\\
&-\kappa\int_0^t\int A_\kappa N_\nq \ A_{\kappa}(\na_L\cdot(N\na_L C))_\nq  dV d\tau.
\end{align}
Recalling the relation \eqref{M_property_ED} and the null condition $U_0^{(2)}=\pa_z\de_y^{-1}\Omega_0=0$, we have that if $\delta\leq \frac{1}{16\pi^2}$, then 
\begin{align}
\frac{1}{2}\|A_\kappa N_{\neq}(t)\|_{L^2}^2 \leq &\frac{1}{2}\|A_\kappa N_{\text{in};\neq}\|_{L^2}^2-\frac{1}{2}\int_0^t \left\|\sqrt{\frac{-\pa_\tau M_\kappa}{  M_\kappa}}A_\kappa N_{\neq}\right\|_{L^2}^2d\tau-\frac{\kappa}{2} \|A_\kappa\sqrt{-\de_L}N_{\neq}\|_{L_t^2L^2}^2\\
&+\bigg|\int_0^t\int  A_\kappa(U_0^{(1)}\ \pa_z N_{\neq})\ A_\kappa N_{\neq}  dVd\tau\bigg| +\bigg|\int_0^t\int  A_\kappa N_{\neq}\ A_\kappa\na_L \cdot (N\ {\na_L^\perp}\de_L^{-1}\Omega_{\neq}) dVd\tau\bigg|\\
&+\bigg|\kappa \int_0^t\int A_\kappa    \pa_y^\tau \left(N_{\neq} \pa_y C_{0}  \right)\ A_\kappa N_{\neq} dVd\tau \bigg| +\bigg|\kappa \int_0^t\int A_\kappa  \na_L\cdot\left( N \na_L C_{\neq}\right)\ A_\kappa N_{\neq} dVd\tau \bigg| \\ 
=:&\frac{1}{2}\|A_\kappa N_{\text{in};\neq}\|_{L^2}^2- {\frac{1}	{2}\int_0^t\norm{\sqrt{\frac{-\pa_\tau M_\kappa}{  M_\kappa}}A_\kappa N_{\neq}}_{L^2}^2d\tau} -\frac{\kappa}{2} \|A_\kappa\sqrt{-\de_L}N_{\neq}\|_{L_t^2L^2}^2\\
&+{T}_{N_\nq;11}+{T}_{N_\nq;12}+{T}_{N_\nq;21}+{T}_{N_\nq;22}.\label{N_T_1T_2}
\end{align}
The estimates of the terms in \eqref{N_T_1T_2} are collected in the following lemmas. 
\begin{lem}\label{lem:N_T_1}
The $T_{N_\nq;11},\ T_{N_\nq;12}$ terms in \eqref{N_T_1T_2} are bounded as follows 
\begin{align}
T_{N_\nq;11}+ T_{N_\nq;12}\leq C\ep^{1/2}(\mathcal{B}_{\om_\nq}^2+\mathcal{B}_{\om_0}^2+\mathcal{B}_{N_\nq}^2  + \mathcal{B}_{N_0}^2 )\mathcal{B}_{N_\nq} .\label{N_T_1}
\end{align} 
Here $C$ is a constant depending only on the regularity level  $s$. 
\end{lem}

\begin{lem}\label{lem:N_T_2}
The ${T}_{N_\nq;21}$ and ${T}_{N_\nq;22}$ terms are bounded as follows
\begin{align}{T}_{N_\nq;2  1}+{T}_{N_\nq;2  2}\leq\frac{\kappa}{8} \int_0^t\|A_\kappa \sqrt{-\de_L}N_\nq\|_{L^2}^2 d\tau + C\kappa^{2/3} (\mathcal{B}_{N_\nq}^2+\mathcal{B}_{N_0}^2)\mathcal{B}_{N_\nq}^2 .\label{N_T_2}
\end{align}
Here $C$ is a constant depending only on the regularity level  $s$.
\end{lem}
Now we complete the proof of \eqref{Conclusion_cell_density_enhanced_dissipation}.
\begin{proof}[Proof of \eqref{Conclusion_cell_density_enhanced_dissipation}]
Combining the decomposition \eqref{N_T_1T_2},  Lemma \ref{lem:N_T_1},  Lemma \ref{lem:N_T_2} and the choice of parameters $0<\kappa\leq \ep=\frac{\kappa}{\nu}\leq 1$, we obtain
\begin{align}
\|A_\kappa &N_\nq(t)\|_{L^2}^2+  \left\|\sqrt{\frac{-\pa_ t M_\kappa}{  M_\kappa}}A_\kappa N_{\neq}\right\|_{L_t^2L^2}^2 + \kappa  \|A_\kappa\sqrt{-\de_L}N_{\neq}\|_{L^2_tL^2}^2\\
\leq &C\|N_{\text{in};\neq}\|_{H^s}^2+
C\ep^{1/2}( \mathcal{B}_{\om_\nq}^2+\mathcal{B}_{\om_0}^2+\mathcal{B}_{N_\nq}^2  +\mathcal{B}_{N_0} ^2 )\mathcal{B}_{N_\nq}+C\ep^{1/2} (\mathcal{B}_{N_\nq}^2+\mathcal{B}_{N_0}^2)\mathcal{B}_{N_\nq}^2. 
\end{align} 
Here $C$ is universal constant depending only on $s$.
The following choices of parameters yields \eqref{Conclusion_cell_density_enhanced_dissipation}:
\begin{align} \label{Choice_of_Const_3}
\mathcal{B}_{N_\nq}^2\geq 4C\|N_{ \mathrm{in};\nq}\|_{H^s}^2,\quad \ep \leq \frac{1}{16 C^2(\mathcal{B}_{\om_\nq}^2+\mathcal{B}_{\om_0}^2+\mathcal{B}_{N_\nq} ^2  + \mathcal{B}_{N_0}^2)^2}. 
\end{align}
\end{proof}
The remaining part of the subsection is devoted to the proof of Lemma \ref{lem:N_T_1} and Lemma \ref{lem:N_T_2}. 
\begin{proof}[Proof of Lemma \ref{lem:N_T_1}]
The estimate of  $T_{N_\nq;11}$ is similar to the estimate of $T_{\om_\nq;12}$ in \eqref{T_1_12}. Hence we will only sketch the estimate. First we note that by the velocity law and the null condition,
\begin{align}
\int \pa_y\de_y^{-1}\om_0\pa_z A_\kappa N_{\neq}\ A_\kappa N_{\neq}dV=\frac{1}{2}\int \pa_y\de_y^{-1}\om_0\ \pa_z \left(A_\kappa N_\nq\right)^2dV=0,
\end{align}
the $T_{N_\nq;11}$-term can be  rewritten as follows
\begin{align}
T_{N_\nq;11}=&\bigg|\int_0^t\int \left( -A_\kappa(\pa_y\de_y^{-1}\om_0 \ \pa_zN_{\neq})+ \pa_y\de_y^{-1}\om_0\ \pa_z A_\kappa N_{\neq}\right)\ A_\kappa N_{\neq}dVd\tau\bigg|\\
 =&C\bigg| \sum_{k\neq0}\int_0^t\iint(M_\kappa(\tau,k,\eta)\lan k,\eta\ran^s-M_{\kappa}(\tau,k, \xi)\lan k,\xi\ran^s)\frac{i(\eta-\xi)}{ (\eta-\xi)^2}\wh \om(\tau,0,\eta-\xi)\\
&\times (ike^{\delta \kappa^{1/3}|k|^{2/3}\tau} \wh N(\tau,k,\xi)) \  \overline{A_\kappa \wh{N}}(\tau,k,\eta)d\xi d\eta d\tau\bigg|.
\end{align}
\ifx 
Application of the commutator estimate \eqref{commutator_estimate} yields that
\begin{align}
|\mathcal{T}_{1;0\neq}|\leq C\sum_{k\neq0}\int_0^t \iint \bigg((1+k^2+\xi^2)^{\frac{s}{2}}+(1+k^2+\eta^2)^{\frac{s}{2}}\bigg)\big |\wh \om(0,\eta-\xi)\big| \big|e^{\delta \kappa^{1/3}|k|^{2/3}\tau}\wh N(k,\xi)\big|
\big|A_\kappa(k,\eta)\overline{\wh{N}}(k,\eta)\big|d\xi d\eta d\tau. 
\end{align}
Now by the $M_\kappa$-estimates \eqref{M_bound}, \eqref{ED_M_Lt2Lzy2}, the product estimate of the multiplier $M_\kappa$ \eqref{M_product_rule_Hs},
\fi Now we observe that this is in the same form as  \eqref{commutator_form}. Hence combining the application of an identical argument as in \eqref{T_om_nq_12}, the enhanced dissipation relation \eqref{ED_M_Lt2Lzy2} and the bootstrap hypotheses \eqref{Hypothesis_cell_density_enhanced_dissipation}, \eqref{Hypothesis_Omega_0} yields that
\begin{align}
 T_{N_\nq;11}\leq C\|A_\nu \om_0 \|_{L^\infty_t L^2}\|A_\kappa N_{\neq}\|_{L_t^2L^2}^2\leq C\mathcal{B}_{\om_0}\mathcal{B}_{N_\nq}^2\ep^{1/2}\kappa^{1 / 6}.
\end{align}
We note that this is consistent with \eqref{N_T_1}.

\myb{
For the $ T_{N_\nq;12}$ term in \eqref{N_T_1T_2}, we apply integration by parts, the Biot-Savart law, and H\"older inequality to obtain the following:
\begin{align}
T_{N_\nq;1 2}=&\bigg|\int_0^t \int \na_L A_\kappa N_\nq \cdot   A_\kappa(\na_L^\perp \de_L^{-1}\om_\nq N ) dVd\tau \bigg| 
\leq  C\|A_\kappa \sqrt{-\de_L}  N_\nq\|_{L_t^2 L^2} \| A_\kappa(\na_L^\perp \de_L^{-1}\om_\nq N)\|_{L_t^2 L^2}.
\end{align}
Now we invoke the product estimate \eqref{A_product_rule_Hs}, and then the $M_\kappa$-estimate \eqref{M_bound} and the elliptic estimate  \eqref{Green_M_t_est} to derive the  following bound
\begin{align}T_{N_\nq;12 }\leq &\|A_\kappa \sqrt{-\de_L}N_\nq\|_{L_t^2L^2}\|A_\nu \na_L^\perp\de_L^{-1}\om_\nq\|_{L_t^2L^2}\|A_\kappa N\|_{L_t^\infty L^2}\\
\leq &\|A_\kappa \sqrt{-\de_L}N_\nq\|_{L_t^2L^2}\left\|A_\nu \sqrt{\frac{-\pa_t M_\nu}{M_\nu}}\om_\nq\right\|_{L_t^2L^2}(\|A_\kappa N_0\|_{L_t^\infty L^2}+\|A_\kappa N_\nq\|_{L_t^\infty L^2}).
\end{align}
Now the bootstrap hypotheses \eqref{Hypothesis_cell_density_enhanced_dissipation}, \eqref{Hypothesis_N_0}, \eqref{Hypothesis_omega_enhanced_dissipation} yields
\begin{align}
 T_{N_\nq ;12}\leq &C\kappa^{-1/2}\mathcal{B}_{N_\nq}\mathcal{B}_{\om_\nq}\ep{\nu^{1/2}}( \mathcal{B}_{N_0}+{\mathcal{B}_{N_\nq}})
\leq C\mathcal{B}_{N_\nq}(\mathcal{B}_{\om_\nq}^2+\mathcal{B}_{N_0}^2+{\mathcal{B}_{N_\nq}^2})\sqrt\ep.
\end{align}} 
Combining the above estimates of $T_{N_\nq;11}$ and $T_{N_\nq;12}$ yields the result \eqref{N_T_1}.
\end{proof}
 
\begin{proof}[Proof of Lemma \ref{lem:N_T_2}]
First we estimate the ${T}_{N_\nq;21}$-term  with integration by parts and the product estimate \eqref{A_product_rule_Hs} as follows:
\begin{align}
{T}_{N_\nq;21}=&\bigg|\kappa \int_0^t\int A_\kappa \na_L N_\nq\cdot A_\kappa (N_\nq \pa_y C_0) dVd\tau\bigg|
\leq{\frac{\kappa}{16} \|A_\kappa \sqrt{-\de_L}N_\nq\|_{L^2_ t L^2}^2} + C\kappa  \|A_\kappa N_\nq \|_{L_t^2L^2}^2\|A_\kappa\pa_y C_0\|_{L_t^\infty L^2}^2.
\end{align}
Now we invoke the chemical gradient estimate \eqref{pa_yC_0_est_Hs}, the enhanced dissipation relation \eqref{ED_M_Lt2Lzy2} and the hypotheses \eqref{Hypothesis_cell_density_enhanced_dissipation}, \eqref{Hypothesis_N_0} to get
\begin{align}
{T}_{N_\nq;21}\leq \frac{\kappa}{16} \int_0^t\|A_\kappa \sqrt{-\de_L}N_\nq\|_2^2 d\tau + C\kappa^{2/3} \mathcal{B}_{N_\nq}^2\mathcal{B}_{N_0}^2.
\end{align}
To estimate the ${T}_{N_\nq;22}$-term, we use the integration by parts and the product estimate \eqref{A_product_rule_Hs}  to obtain that 
\begin{align}
{T}_{N_\nq;22}=\bigg|\kappa \int_0^t\int A_\kappa \na_L N_\nq\cdot A_\kappa (N \na_L C_\nq) dVd\tau\bigg|
\leq \frac{\kappa}{16} \|A_\kappa \sqrt{-\de_L}N_\nq\|_{ L^2_t L^2}^2  + C\kappa  \|A_\kappa N  \|_{L_t^\infty L^2}^2\|A_\kappa\na_L C_\nq\|_{L^2_t L^2}^2.
\end{align}
We use the $M_\kappa$-estimate \eqref{M_bound}, \eqref{M_property_common_dot_M} and the hypotheses \eqref{Hypothesis_cell_density_enhanced_dissipation}, \eqref{Hypothesis_N_0} to obtain the following
\begin{align}
{T}_{N_\nq;22}\leq &\frac{\kappa}{16} \|A_\kappa \sqrt{-\de_L}N_\nq\|_{ L^2_t L^2}^2   + C\kappa  \|A_\kappa N  \|_{L_t^\infty L^2}^2\|A_\kappa\na_L(1-\de_L)^{-1} N_\nq\|_{L^2_t L^2}^2\\
\leq &\frac{\kappa}{16} \|A_\kappa \sqrt{-\de_L}N_\nq\|_{ L^2_t L^2}^2   + C\kappa  \|A_\kappa N  \|_{L_t^\infty L^2}^2\left\|A_\kappa\sqrt{\frac{-\pa_t M_\kappa}{M_\kappa}} N_\nq\right\|_{L^2_t L^2}^2\\
\leq &\frac{\kappa}{16} \|A_\kappa \sqrt{-\de_L}N_\nq\|_{ L^2_t L^2}^2   + C\kappa (\mathcal{B}_{N_\nq}^2+\mathcal{B}_{N_0}^2)\mathcal{B}_{N_\nq}^2 .
\end{align}
Combining the above estimates of $T_{N_\nq;21}$ and $T_{N_\nq;22}$ yields the result \eqref{N_T_2}.
\end{proof}
\subsection{The $z$-average of Cell Density}\label{Sect:Cell_2}
In this section, we prove \eqref{Conclusion_N_0}. The strategy we adopt is to derive an estimate of the $L^2$-norm of $N_0$, and inductively derive higher Sobolev norm bound. 

First we write down the time evolution of the $L^2$ and $H^s$ energy. We consider multiplier $\mathfrak{M}\in\{1,\,\lan\pa_y\ran,...,\lan \pa_y\ran^s\}$. Recalling the equation \eqref{N_0}, we have that 
\begin{align}
\frac{1}{2}\frac{d}{dt}\|\mathfrak{M} N_0\|_2^2=&-\kappa\|\pa_y \mathfrak{M}  N_0\|_2^2-\kappa\int  \mathfrak{M} (\na_L\cdot(N\na_L C))_0\ \mathfrak{M}  N_0dV-\int \mathfrak{M} \left(\na_L\cdot(N\na_L^\perp \de_L^{-1} \om)\right)_0  \ \mathfrak{M}  N_0 dV\\
=&-\kappa\|\pa_y \mathfrak{M}  N_0\|_2^2-\kappa\int  \mathfrak{M} (\pa_y(N_0\pa_y C_0))\ \mathfrak{M}  N_0dV\\
&-\kappa\int  \mathfrak{M} (\na_L\cdot(N_{\neq}\na_LC_{\neq}))_0\ \mathfrak{M} N_0 dV-\int  \mathfrak{M}\left(\na_L\cdot(N\na_L^\perp \de_L^{-1} \om)\right)_0\ \mathfrak{M} N_0 dV.
\end{align}
Next we observe the following relations
\begin{align}
(\na_L\cdot (N_\nq\na_L C_\nq))_0=&(\pa_y^t (N_\nq \pa_y^t C_\nq))_0=\pa_y( N_\nq \pa_y^t C_\nq)_0,\\
\left(\na_L\cdot (N \na_L^\perp\de_L^{-1}\Omega)\right)_0 =& ( \pa_y^t (N\pa_z \de_L^{-1}\om)) _0= \pa_y\left(N_{\neq} \pa_z\de_L^{-1}\Omega_{\neq}\right)_0.
\end{align}
With these, we can rewrite the time evolution of $\|\mathfrak{M}N_0\|_2^2$ as follows
\begin{align}
\frac{1}{2}\frac{d}{dt}\|\mathfrak{M} N_0\|_2^2=&-\kappa\|\pa_y \mathfrak{M} N_0\|_2^2+\kappa\int \mathfrak{M}  \pa_y N_0\  \mathfrak{M} (N_0\pa_y C_0)dV\\
&+\kappa\int \mathfrak{M}\pa_y  N_0\  \mathfrak{M}(N_\nq \pa_y ^t C_\nq)_0dV +\int  \mathfrak{M}\pa_y N_0 \  \mathfrak{M}\left(N_{\neq} \pa_z\de_L^{-1}\Omega_{\neq}\right)_0dV\\
=:&-\kappa\|\pa_y \mathfrak{M} N_0\|_2^2+T_{N_0;0}+T_{N_0;1}+T_{N_0;2}\label{Z_NZ_C_NZ_om} 
\end{align}
The remaining part of the proof is subdivided into several lemmas. 
The first lemma provides estimates for the contributions from the non-zero modes.
\begin{lem}\label{lem:NZ_C_om_est}
There exists a constant $C$, which only depends on the regularity level $s$, such that the following estimates hold
\begin{align}
\int_0^t\|( N_\nq \pa_y^\tau C_\nq)_0\|_{H^s_y}^2d\tau \leq& C \mathcal{B}_{N_\nq}^4 ;\label{NZ_C_est}\\
\int_0^t\|(N_\nq \pa_z\de_L^{-1}\om_\nq)_0\|_{H_y^s}^2d\tau\leq& C\mathcal{B}_{N_\nq}^2\mathcal{B}_{\om_\nq}^2\ep^2\nu.\label{NZ_om_est} 
\end{align}
\end{lem}
\begin{proof}
To prove \eqref{NZ_C_est}, we use Lemma \ref{lem:F_0 and F}, the $M_\kappa$-multiplier bound \eqref{M_bound}, the definition of $A_\kappa$ \eqref{A_N_Omega}, the product estimate \eqref{A_product_rule_Hs}, the elliptic estimate \eqref{Green_M_t_est}, and the bootstrap hypothesis \eqref{Hypothesis_cell_density_enhanced_dissipation} as follows 
\begin{align}\int_0^t\|( N_\nq \pa_y^\tau C_\nq)_0\| _{H^s_y}^2d\tau\leq&C\|\lan \pa_z,\pa_y\ran^s (\pa_y^t C_{\neq} N_{\neq})\|_{L_t^2L^2}^2\leq  C\|A_\kappa (\pa_y^t C_{\neq} N_{\neq}) \|_{L_t^2L^2}^2\nb\\
\leq&  C\norm{A_\kappa \pa_y^t \de_L^{-1}  N_{\neq}}_{L_t^2 L^2}^2\|A_\kappa N_{\neq}\|_{L_t^\infty L^2}^2\\
\leq& C\norm{A_\kappa \sqrt{\frac{-\pa_tM_\kappa}{M_\kappa}}  N_{\neq} }_{L_t^2 L^2}^2\|A_\kappa N_{\neq}\|_{L_t^\infty L^2}^2\leq C \mathcal{B}_{N_\nq}^4.
 \label{NZ_C}
\end{align} 
This concludes the proof of \eqref{NZ_C_est}. Next we prove the estimate \eqref{NZ_om_est}. The idea is identical to the proof \eqref{NZ_C}. The adjustment is that we apply the fact $A_\kappa\approx A_\nu $ \eqref{A_nu_kappa_comparison}, the bootstrap hypotheses \eqref{Hypothesis_cell_density_enhanced_dissipation} and \eqref{Hypothesis_omega_enhanced_dissipation} during the estimate  
\begin{align}
\int_0^t\|(N_\nq \pa_z\de_L^{-1}\om_\nq)_0\|_{H_y^s}^2d\tau\leq &C\|\lan \pa_z,\pa_y\ran^s (N_\nq\pa_z\de_L^{-1}\om_\nq)\|_{L_t^2L^2}^2\leq C\|A_\kappa(N_\nq\pa_z\de_L^{-1}\om_\nq)\|_{L_t^2L^2}^2\\
\leq& C  \|A_\kappa N_\nq\|_{L_t^2L^2}^2\|A_\kappa\pa_z\de_L^{-1}\om_\nq\|_{L_t^\infty L^2}^2
\leq C  \|A_\kappa N_\nq\|_{L_t^\infty L^2}^2\|A_\nu\pa_z\de_L^{-1}\om_\nq\|_{L_t^2 L^2}^2\\
\leq&C\|A_\kappa N_{\neq}\|_{L_t^\infty L^2}^2 \norm{A_\nu \sqrt{\frac{-\pa_tM_\nu}{M_\nu}}  \om_{\neq} }_{L_t^2 L^2}^2\leq C \mathcal{B}_{N_\nq}^2 \mathcal{B}_{\om_\nq}^2\ep^2\nu.
\end{align}
This concludes the proof of \eqref{NZ_om_est}.

\end{proof}

Next, we prove the following $L^2$ estimate for $N_0$.
\begin{lem}\label{Lem:N_0_L2_estimate} If the hypotheses \eqref{Hypothesis_cell_density_enhanced_dissipation}, \eqref{Hypothesis_N_0} hold on $[0,T_{\star}]$, then the solution $N_0$ is bounded uniformly in $L^2$
\begin{align}\label{N_0_L_2}
\|N_0\|_{L_y^2(\rr)}\leq C_{L^2}:=\|N_{\mathrm{in};0}\|_2+CM^{3/2}+C(\mathcal{B}_{N_\nq}^2+\mathcal{B}_{N_\nq}\mathcal{B}_{\om_\nq} \ep^{1/2}), \quad \forall t\in [0,T_\star].
\end{align}
{Here $M=\|n\|_{L^1}=|\Torus|\int_{\rr}N_{\mathrm{in};0}(y)dy$ is the conserved total mass of the cell density.}
\end{lem}

\begin{proof}
Before estimating the norm $\|N_0\|_2$, we collect the $L^1(\rr)$ bound of $N_0$. Since the solution $N(t,z,y)$ to the equation \eqref{PKS_NS_zy_coordinate} is \myb{positive by Theorem \ref{thm:local_existence},} we have that the function $N_0(y)=\frac{1}{2\pi}\int_{\mathbb{T}}N(t,z,y)dz$ is positive. Moreover, the total integral of $N_0$ is preserved due to the divergence form of the equation (\ref{N_0}),
\be
\int N_0(t)dy=\int N_0(0,y)dy,
\ee
which, together with the positivity of $N_0$, implies that
\begin{align}\label{N_0 L1}
\|N_0(t)\|_{L^1(\rr)}\equiv \frac{M}{2\pi},\quad \forall t\in[0,T_\star].
\end{align}


Next we estimate each component in $\eqref{Z_NZ_C_NZ_om}_{\mathfrak{M}=1}$, which describes the time evolution of $\|N_0(t)\|_2^2$.  
Combining the information \eqref{N_0 L1}, \eqref{pa_yC_0_est_Linf}, the $T_{N_0;0}$-term in \eqref{Z_NZ_C_NZ_om} can be  estimated as follows
\begin{align}
T_{N_0;0}\leq &\kappa\int |\pa_y N_0|\ |\pa_y C_0N_0|dy
\leq \frac{1}{4 }\kappa\|\pa_y N_0\|_2^2+ \kappa\|\pa_y C_0\|_{\infty}^2\|N_0\|_2^2 
\leq \frac{1}{4}\kappa\|\pa_y N_0\|_2^2+ C\kappa M^2\|N_0\|_{2}^2.\label{Z}
\end{align}
The $ T_{N_0;1},\,T_{N_0;2}$ terms can be estimated as follows
\begin{align}
T_{N_0;1}(t)+T_{N_0;2}(t)\leq& \frac{1}{2}\kappa\|\pa_y N_0\|_{L_y^2}^2(t)+C\left(\kappa \|N_\nq \pa_y^t C_\nq\|_{L_{z,y}^2}^2(t)+\frac{1}{\kappa}\|N_\nq \pa_z \de_L^{-1}\om_\nq\|_{L_{z,y}^2}^2(t)\right)\\
=:&\frac{1}{2}\kappa\|\pa_y N_0\|_{L_y^2}^2(t)+\frac{d}{dt} G_{L^2}(t),\quad G_{L^2}(0)=0.
\end{align}
By Lemma \ref{lem:NZ_C_om_est},
\begin{align} 
G_{L^2}(t)\leq C \mathcal{B}_{N_\nq}^4+C \mathcal{B}_{N_\nq}^2\mathcal{B}_{\om_\nq}^2\ep,\quad \forall t\in[0, T_\star].\label{G_L^2_bound}
\end{align}
Combining the above estimates \eqref{Z}, \eqref{G_L^2_bound} and the relation \eqref{Z_NZ_C_NZ_om}, we have
\begin{align}\label{energy estimate zero mode}
\frac{1}{2}\frac{d}{dt}\int_\rr N_0 ^2dy\leq-\frac{1}{4}\kappa\|\pa_y N_0\|_2^2+C\kappa M^2\|N_0\|_{2}^2+\frac{d}{dt}G_{L^2}(t).
\end{align} 
Now we try to get an estimate on the $L^2$ norm from (\ref{energy estimate zero mode}). Applying the following Nash inequality
\be
\|f\|_{L^2(\rr)}^{3}\leq C  \|f\|_{L^1(\rr)}^{2}\|\pay f\|_{L^2(\rr)}
\ee
yields
\be
-\|\pa_yf\|_{2}^2\leq - \frac{\|f\|_{2}^{6}}{C\|f\|_{1}^{4}}.
\ee
Combining this with the estimate \eqref{energy estimate zero mode}, we have that
\begin{align}
 \frac{d}{dt}\left(\| N_0\|_2^2-{G_{L^2}(t)}\right)\leq&-\frac{1}{2 C}\kappa\|\pa_y N_0\|_2^2+\kappa C M^2\|N_0\|_{2}^2 
\leq-  \kappa\frac{\|N_0\|_{2}^{6}}{CM^{4}}+\kappa  CM^2\|N_0\|_{2}^2 \\
\leq&-\frac{1}{CM^4}\kappa\|N_0\|_{2}^{2}\bigg(\|N_0\|_{2}^{4}-CM^6\bigg)\\
\leq &-\frac{1}{CM^4}\kappa\|N_0\|_{2}^{2}\bigg(\|N_0\|_{2}^{2}-{G_{L^2}(t)}-CM^3\bigg)\bigg(\|N_0\|_{2}^{2}+CM^3\bigg) .\label{N_0L2_est}
\end{align}
We can see that the $\|N_0\|_2$ is bounded uniformly in $[0,T_\star]$ in the sense that
\begin{align}
\|N_0(t)\|_2\leq C_{L^2}=\|N_{\mathrm{in};0}\|_2+CM^{3/2}+C(\mathcal{B}_{N_\nq}^2+\mathcal{B}_{N_\nq}\mathcal{B}_{\om_\nq} \ep^{1/2}),\quad \forall t\in [0,T_\star],
\end{align}
which is \eqref{N_0_L_2}.
\end{proof}

Next we try to use the information on $\|N_0\|_2$ to get the bound on higher $H^s$ norm.

\begin{lem}
Assuming the hypotheses \eqref{Hypothesis_cell_density_enhanced_dissipation}, \eqref{Hypothesis_omega_enhanced_dissipation} hold on $[0,T_\star]$, then the $H^s$ norm of the solution $N_0$ to \eqref{N_0} is bounded:
\begin{align}
\|N_0(t)\|_{H^s}\leq C_{H^s} (\|N_{\mathrm{in}}\|_{H^s},M, \mathcal{B}_{N_\nq},\sqrt{\ep}\mathcal{B}_{\om_\nq}) , \quad\forall t\in [0,T_\star].
\end{align}
{Here $M=\|n\|_{L^1}=|\Torus|\int_{\rr}N_{\mathrm{in};0}(y)dy$ is the conserved total mass of the cell density.} 
Moreover, if 
\begin{align}
\ep\leq \frac{1}{\mathcal{B}_{\om_\nq}^2},\label{Choice_of_Const_4}
\end{align}
then
\begin{align}
\|N_0\|_{H^s}\leq C_{H^s}(\|N_{\mathrm{in}}\|_{H^s},M, \mathcal{B}_{N_\nq}).\label{N_0_H_s_est}
\end{align}
\end{lem}

\begin{proof} 
We estimate the $H^i$-norms ($i\in\{1,2,...,s\}$) in the inductive fashion. Assume that we have the following estimate for some $1\leq i\leq s  $,
\begin{align}
\|N_0\|_{H^{i-1}}\leq C_{H^{i-1}}(\|N_{\text{in};0}\|_{H^s},M,\mathcal{B}_{N_\nq},\sqrt{\ep} \mathcal{B}_{\om_\nq}).\label{Induction_assumption}
\end{align}
We would like to prove that 
\begin{align}\label{induction_conclusion}
\|N_0\|_{H^i}\leq C_{H^{i}}(\|N_{\text{in};0}\|_{H^s},M,\mathcal{B}_{N_\nq},\sqrt{\ep} \mathcal{B}_{\om_\nq}). 
\end{align}
Similar to the $L^2$ case, we estimate each term in $\eqref{Z_NZ_C_NZ_om}_{\mathfrak{M}=\lan\pa_y\ran^i}$. 

The $T_{N_0;1},\, T_{N_0;2}$ terms can be estimated as follows
\begin{align}
T_{N_0;1}(t)+T_{N_0;2}(t)\leq& \frac{1}{2}\kappa\|\pa_y \lan \pa_y\ran^i N_0\|_{L^2}^2(t)+C\left(\kappa \|N_\nq \pa_y^t C_\nq\|_{H^i}^2(t)+\frac{1}{\kappa}\|N_\nq \pa_z \de_L^{-1}\om_\nq\|_{H^i}^2(t)\right)\\
=:&\frac{1}{2}\kappa\|\pa_y\lan\pa_y\ran^i N_0\|_{L^2}^2(t)+\frac{d}{dt} G_{H^i}(t),\quad G_{H^i}(0)=0.\label{NZ_C,NZ_om_Hi}
\end{align}
By Lemma \ref{lem:NZ_C_om_est},
\begin{align} 
G_{H^i}(t)\leq C \mathcal{B}_{N_\nq}^4+C \mathcal{B}_{N_\nq}^2\mathcal{B}_{\om_\nq}^2\ep,\quad \forall t\in[0, T_\star].\label{G_Hi_bound}
\end{align}

The $T_{N_0;0}$ term can be estimated using H\"{o}lder's inequality, product estimate of Sobolev functions $f,g\in H^i(\rr),\,i\geq 1$, Gagliardo-Nirenberg inequality, and the chemical gradient estimates \eqref{pa_yC_0_est_Linf}, \eqref{pa_yC_0_est_Hs}, as follows,
\begin{align}
T_{N_0;0}\leq&\kappa C\|\pa_y \lan \pa_y \ran^i N_0\|_{L^2}\|\lan \pa_y \ran^i(N_0\pa_y C_0)\|_{L^2}
\leq\kappa C\|\pa_y\lan \pa_y\ran^{i} N_0\|_{L^2}\|\lan \pa_y \ran^i\pa_y C_0\|_{L^2}\|\lan \pa_y\ran^i N_0\|_{L^2}\\
\leq&\frac{1}{4}\kappa\|\pa_y\lan\pa_y\ran^{i } N_0\|_{L^2}^2+C\kappa\|\lan \pa_y\ran^{i-1} N_0\|_{L^2}^2\|\lan \pa_y\ran^{i}N_0\|_{L^2}^2 
\leq\frac{1}{4}\kappa\|\pa_y\lan\pa_y\ran^{i } N_0\|_{L^2}^2+C\kappa C_{H^{i-1}}^2\|\lan \pa_y\ran^{i}N_0\|_{L^2}^2.\label{pa s f 0 time derivative}
\end{align}
Applying the following Gagliardo-Nirenberg inequality
\begin{align}
\|\pa_y^i f\|_{2}\leq C\|f\|_{2}^{\frac{1}{i+1}}\|\pa_y^{i+1}f\|_{2}^{\frac{i}{i+1}}
\end{align}
yields that
\begin{align}
-\|\pa_y^{i+1}N_0\|_{2}^2\leq -\frac{\|\pa_y^i N_0\|_{2}^{\frac{2i+2}{i}}}{C\|N_0\|_{2}^{\frac{2}{i}}}\leq-\frac{\|\pa_y^i N_0\|_{2}^{\frac{2i+2}{i}}}{CC_{L^2}(M,\mathcal{B}_{N_\nq},\sqrt{\ep}\mathcal{B}_{\om_\nq})^{\frac{2}{i}}}.\label{pa s f 0 time first term}
\end{align}
Combining \eqref{NZ_C,NZ_om_Hi}, \eqref{pa s f 0 time derivative} and  \eqref{pa s f 0 time first term} yields
\begin{align}
 \frac{d}{dt}\|\lan\pa_y\ran ^i N_0\|_{2}^2\leq-  \kappa\frac{\|\pa_y^i N_0\|_{2}^{2+\frac{2}{i}}}{4 C C_{L^2}^{\frac{2}{i}}}+\kappa C  C_{H^{i-1}}^2\|\lan\pa_y\ran ^i N_0\|_{2}^2+\frac{d}{dt} G_{H^i}(t).\label{ODE_pa_y_sN_0}
\end{align}
Here $C_{L^2}$ and $C_{H^{i-1}}$ only depend on $\|N_{\text{in}}\|_{H^s}, \ M, \ \mathcal{B}_{N_\nq}$ and $\sqrt{\ep}\mathcal{B}_{\om_\nq}$. Now we apply an ODE argument similar to the one in \eqref{N_0L2_est}  to derive that the quantity $\|N_0\|_{H^i}$ is uniformly bounded on the time interval $[0,T_\star]$, i.e., \eqref{induction_conclusion}. \myb{We first recall the relation $ C_i ^{-1}(\|\pa_y^i N_0\|_{L^2}^2+\|N_0\|_{H^{i-1}}^2)\leq\|\lan \pa_y \ran^i N_0\|_{L^2}^2\leq C_i (\|\pa_y^i N_0\|_{L^2}^2+\|N_0\|_{H^{i-1}}^2)$ with $ C_i\geq 1$, and the $G_{H^i}$-estimate \eqref{G_Hi_bound}. Then we distinguish between two scenarios: 
\begin{align}
a) \, \|\pa_y^i N_0(t)\|_{L^2}^2< 4 C_i(G_{H^i}(t)+C_{H^{i-1}}^2);\ \ 
b) \, \|\pa_y^i N_0(t)\|_{L^2}^2\geq 4 C_i(G_{H^i}(t)+C_{H^{i-1}}^2).\label{Case}
\end{align}
In the first scenario, the estimate \eqref{induction_conclusion} is direct. Assume that on some open time intervals, the estimate $b)$ in \eqref{Case} holds. Then the time evolution \eqref{ODE_pa_y_sN_0} implies that there exists a constant $C$ such that 
\begin{align}
\frac{d	}{dt}(\|\lan \pa_y\ran^i N_0(t)\|_2^2-G_{H^{i}}(t))\leq -\kappa\frac{(\|\lan \pa_y\ran^i N_0(t)\|_2^2-G_{H^i}(t))^{\frac{i+1}{i}}}{CC_{L^2}^{2/i}}+\kappa C C_{H^{i-1}}^2(\|\lan \pa_y\ran^i N_0(t)\|_2^2-G_{H^{i}}(t)).
\end{align}
Hence by \eqref{G_Hi_bound}, we have that 
\begin{align}
\|\lan \pa_y \ran^i N_0(t)\|_{L^2}^2\leq C(\|N_{\mathrm{in}}\|_{H^s},M,C_{L^2}, C_{H^{i-1}}, \mathcal{B}_{N_\nq},\sqrt{\ep} \mathcal{B}_{\om_\nq}). 
\end{align}
This is consistent with \eqref{induction_conclusion}.} Hence, the induction estimate \eqref{induction_conclusion} is established. Since $i$ ranges in $\{1,...,s\}$, we end up with the following
\begin{align}\label{zero mode control}
\|  N_0(t)\|_{H^s}\leq C_{H^s}(\|N_{\mathrm{in}}\|_{H^s},M,\mathcal{B}_{N_\nq},\sqrt{\ep}\mathcal{B}_{\om_\nq}), \quad \forall t\in [0,T_\star]. 
\end{align}
This concludes the proof.
\end{proof}
\begin{proof}[Proof of \eqref{Conclusion_N_0}]
To prove \eqref{Conclusion_N_0}, we choose \eqref{Choice_of_Const_4}, and 
\begin{align}
\mathcal{B}_{N_0}\geq 4 C_{H^s}(\|N_{\mathrm{in}}\|_{H^s},M, \mathcal{B}_{N_\nq}). \label{Choice_of_Const_5}
\end{align}
Here $C_{H^s}$ is defined in \eqref{N_0_H_s_est}.
\end{proof}

\begin{proof}[Proof of Proposition \ref{pro_bootstrap}]
It is enough to show that the choice of parameters  \eqref{Choice_of_Const_1}, \eqref{Choice_of_Const_2}, \eqref{Choice_of_Const_3}, \eqref{Choice_of_Const_4}, and \eqref{Choice_of_Const_5} are consistent. We choose the parameters 
\begin{align}
\mathcal{B}_{N_\nq}^2:=&C_1\|N_{\text{in}}\|_{H^s}^2,\quad  \mathcal{B}_{N_0}^2:=C_2C_{H^s}^2(\|N_{\mathrm{in}}\|_{L^1\cap H^s},\mathcal{B}_{N_\nq}),\\
\quad \mathcal{B}_{\om_\nq}^2:=&C_3(\mathcal{B}_{N_\nq}^4+\mathcal{B}_{N_0}^4),\quad \mathcal{B}_{\om_0}^2:= C_4( \mathcal{B}_{\om_\nq}^2+\mathcal{B}_{N_\nq}^4).\label{Choice_of_bootstrap_const}
\end{align}
Here $C_1,C_2,C_3, C_4$ are constants depending only on the universal constants appeared in  \eqref{Choice_of_Const_1}, \eqref{Choice_of_Const_2}, \eqref{Choice_of_Const_3}, \eqref{Choice_of_Const_4}, and \eqref{Choice_of_Const_5}. Now we summarize the choice of $\ep=\kappa/\nu$:
\begin{align}
 \ep\leq\ep_0:= \frac{1}{C_5(\mathcal{B}_{\om_\nq}^4+\mathcal{B}_{\om_0}^4+\mathcal{B}_{N_\nq}^4+\mathcal{B}_{N_0}^4)}.\label{Choice_of_ep}
\end{align} 
Here $C_5$ is a constant depending only on the universal constants appeared in  \eqref{Choice_of_Const_1}, \eqref{Choice_of_Const_2}, \eqref{Choice_of_Const_3}, \eqref{Choice_of_Const_4}. This concludes the proof of Proposition \ref{pro_bootstrap}.
\end{proof}

\appendix
\section{}
\subsection{Miscellaneous}
\begin{lem}\label{lem:F_0 and F}
Let $F\in L^p(\mathbb{T}\times\rr), \, G\in H^s(\mathbb{T}\times\rr)$, and $
F_0(y):=\frac{1}{|\mathbb{T}|}\int_{\mathbb{T}}F(z,y)dz,\ 
G_0(y):=\frac{1}{|\mathbb{T}|}\int_{\mathbb{T}}G(z,y)dz.$ The following estimates hold:
\begin{align}\label{F_0 and F}
\|F_0\|_{L^p(\rr)}\leq& \|F\|_{L^p(\mathbb{T}\times\rr)},\quad 1\leq p\leq\infty,\\
\|\lan \pa_y\ran^s G_0\|_{L^2(\rr)}\leq&C\|\lan \pa_z,\pa_y\ran^s G\|_{L^2(\mathbb{T}\times\rr)}.
\end{align}
\end{lem}

\begin{proof}
Applying the H\"{o}lder's inequality yields that for $p\in(1,\infty)$,
\be\ba
\|F_0\|_{L^p(\rr)}=&\left(\int_\rr \bigg|\frac{1}{2\pi}\int_{\mathbb{T}}Fdz\bigg|^pdy\right)^{1/p}\leq\left(\int_\rr \left(\left(\int_{\mathbb{T}}|F|^pdz\right)^{1/p}\left(\int_{\mathbb{T}}(2\pi)^{-p'} dz\right)^{1/p'}\right)^pdy\right)^{1/p}\\
\leq&\left(\int_\rr \int_{\mathbb{T}}|F|^pdzdy\right)^{1/p}=\|F\|_{L^p(\mathbb{T}\times\rr)}.
\ea\ee
The proof in the $p=1,\infty$ cases are variants of the argument above. 
Applying the Fourier transform and the Plancherel equality yields
\be\ba
\|\lan \pa_y\ran^s F_0\|_{L^2(\rr)}^2=C\int_\rr \lan\eta\ran^{2s}|\widehat{F_0}|^2(\eta)d\eta
\leq{C \sum_k \int_\rr \lan k,\eta\ran^{2s}|\widehat{F}|^2(k,\eta)d\eta}=C\|\lan \pa_z,\pa_y\ran^s F\|_{L^2(\mathbb{T}\times\rr)}^2.
\ea\ee
This finishes the proof of the lemma.
\end{proof}

\begin{lem}[Chemical gradient estimates]
 Consider solution $C_0$ to the equation $(1-\pa_{yy})C_0=N_0$. The following estimates hold
 \begin{align}\label{pa_yC_0_est_Linf}
 \|\pa_y C_0\|_{L^\infty_y}\leq &\|N_0\|_{L^1_y};\\
 \|\pa_y C_0\|_{H_y^s}\leq &\|N_0\|_{H_y^{s-1}}, \, s\in \mathbb{N}_+.\label{pa_yC_0_est_Hs}
 \end{align}
 \end{lem}
 \begin{proof}
 To prove the \eqref{pa_yC_0_est_Linf}, we use the explicit solution formula to get
\begin{align} 
\|\pa_y C_0\|_{L_y^\infty}=C{\norm{\int_\rr \frac{y-y'}{|y-y'|}e^{-|y-y'|}N_0(y')dy'}_{L_y^\infty}}\leq C\|N_0\|_{L^1_y}.
\end{align}

 To prove the \eqref{pa_yC_0_est_Hs}, we use the Fourier transform and the Plancherel equality
\begin{align}
\|\pa_y C_0\|_{H^s}^2=C\int \lan \eta\ran^{2s} \bigg|\frac{i\eta}{1+|\eta|^2}\wh N_0\bigg|^2d\eta\leq C\int \lan \eta\ran^{2(s-1)}|\wh N_0|^2d\eta\leq C\|N_0\|_{H^{s-1}}^2.
\end{align}
\end{proof}
Next we provide a sketch of the proof of Theorem \ref{thm:local_existence}.
\begin{proof}[Proof of Theorem \ref{thm:local_existence}]
Standard energy argument yields the local-in-time $\|N\|_{H_{z,y}^s},\|\om\|_{H_{z,y}^s}$-estimates. We omit further details for the sake of brevity. 

{The justification of the positivity  of the cell density, i.e.,   $N\geq 0$, is as follows. We use the technique from  \cite{CarrilloCastroYaoZeng21}\myr{, page 9 (end of section 2)}. The explicit argument is as follows.  We consider the family of convex positive  functions $j_\ep$ which approximates $j(s)=\max\{-s,0\}$.  We can  choose the $j_\ep$'s to be monotonically increasing to $j$. These functions $j_\ep'=j'$ in $[-\ep,0]^c$ and $0\leq j_\ep''\leq 2\ep^{-1}$ in $[-\ep, 0]$. So $J_\ep(s):=\int_0^sj_\ep''(\sigma)\sigma d\sigma$ satisfies $|J_\ep(s)|\leq 2|s|$ for $\ep\in(0,1)$, and $\lim_{\ep\ra 0^+}J_{\ep}(s)=0$ for all $s$. Now we have that 
\begin{align*}
\frac{d}{dt}\int_{\Torus\times \rr}j_\ep(N(t,z,y))dV 
=&\int_{\Torus\times\rr}j_\ep'(N)(\kappa\de_L N-U\cdot\na_L N -\kappa \na_L\cdot (N\na_L C))dV\\
=&-\kappa\int_{\Torus\times\rr}j_\ep''(N)( |\pa_{z} N|^2+|\pa_{y}^tN |^2)dV-\int_{\Torus\times \rr}U\cdot\na_L (j_\ep(N)) dV\\
&+\kappa\int_{\Torus\times \rr} j_\ep''(N)(N\pa_{z}N\pa_{z} C +N\pa_{y}^tN \pa_{y}^tC)dV\\
\leq & \kappa\int_{\Torus\times\rr}  \pa_{z}\left(\int_0^{N}j_\ep''(s)sds\right)\pa_{z}C dV+\kappa\int_{\Torus\times\rr} \pa_{y}^t\left(\int_0^{N }j_\ep''(s)sds\right)\pa_{y}^tC dV\\
=& \kappa\int_{\Torus\times\rr}\left(\int_0^{N }j_\ep''(s)sds\right) (-\de_L C)dV
=\int_{\Torus\times\rr}\left(\int_0^{N }j_\ep''(s)sds\right) (N-C) dV. 
\end{align*}  
Therefore,
\begin{align}
\int_{\Torus\times \rr}j_{\ep}(N(t ))dV-\int_{\Torus\times \rr}j_{\ep}(N(0))dV=\int_0^t\int_{\Torus\times\rr}\left(\int_0^{N(\tau)}j_\ep''(s)sds\right) (N-C) dVd\tau. 
\end{align}
Since $|J_\ep(s)|\leq 2|s|$, we can use $4(N^2+C^2)$ as the dominator and invoke the Dominated convergence theorem and Monotone convergence theorem to get for all $[0,t]$ on which $L^2$ is bounded, 
\begin{align}
\|N_-(t)\|_{L^1_{z,y}}=&\int_{\Torus\times\rr}j(N(t))dV=\lim_{\ep\ra 0^+}\int_{\rr^2}j_\ep(N(t))dV=\lim_{\ep\ra 0^+}\int_0^t\int_{\Torus\times\rr}\left(\int_0^{N(\tau)}j_\ep''(s)sds\right) (N-C) dVd\tau\\
=&\int_0^t\int_{\Torus\times\rr}\lim_{\ep\ra 0^+}\left(\int_0^{N(\tau,X)}j_\ep''(s)sds\right) (N-C) dVd\tau=0.
\end{align}
As a result, $N\geq 0.$}
\end{proof}
\subsection{Fourier Multipliers}
In this section, we summarize the  properties of the Fourier multipliers that we employ. First, we collect some basic properties of the multipliers $M_\kappa, M_\nu$ defined in \eqref{M_N_Omega}.
\begin{lem}\label{lem:multipliers}
For $\iota\in \{\kappa,\nu\}$, the following  properties  for $M_\iota$ hold
\begin{subequations}\label{M_property_common}
\begin{align}
M_{\iota}(t,k,\eta)&=\pi^2,\quad |k|\notin(0,\iota^{-1/2}] ;\label{M_1}\\
\frac{9}{4}\pi^2\geq& M_\iota(t,k,\eta)\geq \frac{\pi^2}{4}; \label{M_bound}\\
- {\pa_t M_\iota(t,k,\eta)} \geq& \frac{\pi}{2}\frac{|k|^2}{|k|^2+|\eta-kt|^2},\quad k\neq 0;\label{M_property_common_dot_M}\\
 | {\pa_\eta M_\iota(t,k,\eta)}  |\leq &\frac{4\pi}{|k|},\quad k\neq 0, \label{M_property_common_pa_eta}\\
 \frac{\sqrt{-\pa_t M_\iota(t,k,\eta)}}{\sqrt{-\pa_t M_\iota(t,k,\xi)}}\leq&  {2} (1+|\eta-\xi|^2)^{\frac{1}{2}},\quad |k|\in (0,\iota^{-1/2}]\label{M_5}. 
\end{align}
\end{subequations}

\ifx{\bf Here we also need 
\begin{align*}
\pa_t w_k=\frac{k^2}{k^2+|\eta-kt|^2}w_k.
\end{align*}
It seems to me that we can replace $M_\star$ by $\frac{M_{\nu\, or\, \kappa}}{w_k}$.}
\fi 
Moreover, the multipliers $M_\kappa, M_\nu$ have the enhanced dissipation properties
\begin{align}\frac{1}{3\pi}\iota^{1/3} {|k|^{2/3}}\leq&-\frac{\pa_t M_\iota}{ M_\iota}(t,k,\eta)+\iota (|k|^2+|\eta-kt|^2), \quad \iota\in\{\kappa,\nu\}.\label{M_property_ED}
\end{align}
\end{lem}
  \begin{proof}
 First of all, the first two inequalities \eqref{M_1}, \eqref{M_bound} are consequences of the definitions \eqref{W_kappa}, \eqref{W_nu}, \eqref{W} and the boundedness of the $\arctan$-function. The time derivative of the multiplier $\mathcal{W}$ reads as follows 
\begin{align} \label{pa_t_W}
 \pa_t \mathcal {W}( t, k,\eta)=&-\frac{|k|^2}{|k|^2+|\eta-kt|^2},\quad \, k\neq 0.
\end{align} 
Hence combining this expression, and the bounds $W_\iota\geq\pi/2,\quad \pa_tW_\iota\leq 0$, we have that
\begin{align}
-\pa_t M_\iota =- \mathcal{W}\pa_t W_\iota- W_\iota\pa_t \mathcal{W}\geq\frac{\pi}{2}\frac{|k|^2}{|k|^2+|\eta-kt|^2},\quad k\neq 0.
\end{align} This completes the proof of \eqref{M_property_common_dot_M}.

To prove \eqref{M_property_common_pa_eta}, we observe that 
\begin{align}
|\pa_\eta M_\iota(t,k,\eta)|\leq&|\mathcal{W}\pa_\eta W_\iota |+| W_\iota \pa_\eta\mathcal{W}|
\leq  2\pi\ \mathbf{1}_{0<|k|\leq \iota^{-1/2}}\frac{1}{|k|}\frac{\iota^{1/3}|k|^{2/3}}{1+\iota^{2/3}|k|^{4/3}|t-\frac{\eta}{k}|^2}+2\pi\ \mathbf{1}_{k\neq 0}\frac{1}{|k|}\frac{|k|^2}{|k|^2+|\eta-kt|^2}\\
\leq &\frac{4\pi}{|k|}\mathbf{1}_{k\neq0}.
\end{align}
This concludes the proof of \eqref{M_property_common_pa_eta}. 

Next, we prove \eqref{M_5}. Direct computation yields that 
\begin{align}\label{pa_t_W_iota}
\pa_t W_\iota (t,k,\eta)=&-\frac{\iota^{1/3}|k|^{2/3}}{1+\iota^{2/3}|k|^{4/3}|t-\frac{\eta}{k}|^2}\mathbf{1}_{ |k|\in(0, \iota^{-1/2}]},\quad \iota\in \{\kappa,\nu\}.
\end{align}
For the wave number ranging in $|k|\in (0,\iota^{-1/2}]$, we invoke the expressions \eqref{pa_t_W}, \eqref{pa_t_W_iota} to obtain the  following   estimates of the quotient 
\begin{align}
\frac{-\pa_t M_\iota(t,k,\eta)}{-\pa_t M_\iota(t,k,\xi)}\leq& \frac{1+|t-\frac{\xi}{k}|^2}{1+|t-\frac{\eta}{k}|^2}+\frac{1+\iota^{2/3}|k|^{4/3}|t-\frac{\xi}{k}|^2}{1+\iota^{2/3}|k|^{4/3}|t-\frac{\eta}{k}|^2} \\
\leq&  \frac{1+2|t-\frac{\eta}{k}|^2+2|\frac{\eta-\xi}{k}|^2}{1+|t-\frac{\eta}{k}|^2}+\frac{1+2\iota^{2/3}|k|^{4/3}|t-\frac{\eta}{k}|^2+2\iota^{2/3}|k|^{4/3}|\frac{\eta-\xi}{k}|^2}{1+\iota^{2/3}|k|^{4/3}|t-\frac{\eta}{k}|^2}\\
\leq& 4+4\frac{|\eta-\xi|^2}{|k|^2}.
\end{align}
This yields \eqref{M_5}.

Finally, we prove \eqref{M_property_ED}. We recall the expressions \eqref{pa_t_W},  \eqref{pa_t_W_iota} and the bound $W_\iota\in[\pi/2,3\pi/2]$. Combining these ingredients  yields that 
\begin{align}
\frac{-\pa_t M_\iota} {M_\iota}+\iota (|k|^2+|\eta-kt|^2)\geq&\frac{ -\mathcal{W}\pa_t W_\iota  }{M_\iota}+\iota(|k|^2+|\eta-kt|^2)\\
\geq&\mathbf{1}_{0<|k|\leq \iota^{-1/2}}\ \frac{\iota^{1/3}|k|^{2/3}}{1+\iota^{2/3}|k|^{4/3}|t-\frac{\eta}{k}|^2}\ \frac{ 2}{3\pi}+\iota|k|^2\left(1+\bigg|t-\frac{\eta}{k}\bigg|^2\right).
\end{align}
There are two regimes for the wave number $k$: $|k|\in(0,\iota^{-1/2}]$ or $|k|\notin (0,\iota^{-1/2}]$. If $|k|\notin (0,\iota^{-1/2}]$, then it can be checked that $\iota|k|^2\geq \iota^{1/3}|k|^{2/3}$. Hence the result \eqref{M_property_ED} is ensured. On the other hand, if $|k|\in(0, \iota^{-1/2}]$, we estimate the above expression as follows
\begin{align}
\frac{-\pa_t M_\iota}{ M_\iota}&+\iota (|k|^2+|\eta-kt|^2)\\
\geq&  \mathbf{1}_{|t-\frac{\eta}{k}|\leq \iota^{-1/3}|k|^{-2/3}} \ \frac{\iota^{1/3}|k|^{2/3}}{1+\iota^{2/3}|k|^{4/3}|t-\frac{\eta}{k}|^2}\frac{ 2}{3\pi } +\mathbf{1}_{|t-\frac{\eta}{k}|> \iota^{-1/3}|k|^{-2/3}}\ \iota|k|^2\left(1+\bigg|t-\frac{\eta}{k}\bigg|^2\right)
\geq\frac{1}{3\pi}\iota^{1/3}|k|^{2/3}.
\end{align}
This concludes the proof of the lemma.
\end{proof}

The following lemma is a natural consequence of Lemma \ref{lem:multipliers}.
\begin{lem} For any function $f_{\neq}\in H^s$ with vanishing mean $\frac{1}{|\Torus|}\int_\Torus f_\nq dz\equiv 0$, the following estimates hold for $\iota\in\{\kappa,\nu\}$,
\begin{align}
\int_0^t \|A_\iota f_\nq\|_{L^2}^2d\tau \leq& C\iota^{-1/3}\int_0^t\left\|A_\iota \sqrt{\frac{-\pa_\tau M_\iota}{M_\iota}}f_\nq\right\|_{L^2}^2+\iota\|A_\iota\sqrt{-\de_L}f_\nq\|_{L^2}^2d\tau;\label{ED_M_Lt2Lzy2}\\
\|A_\iota\na_L \de_L^{-1} f_\nq\|_{L^2}&+
\|A_\iota\na_L (1-\de_L)^{-1} f_\nq\|_{L^2}\leq \myb{C\left\|A_\iota\frac{1}{|\pa_z|}\sqrt{\frac{-\pa_t M_\iota}{M_\iota}} f_\nq\right\|_{L^2}}.\label{Green_M_t_est}
\end{align}
Here the constant $C$ is universal.
\end{lem}
\begin{proof}
The inequality \eqref{ED_M_Lt2Lzy2} follows from the property \eqref{M_property_ED} and the Plancherel equality. The inequality \eqref{Green_M_t_est} follows from the properties  \eqref{M_bound},  \eqref{M_property_common_dot_M} and the Plancherel equality.
\end{proof}

The product rule for the multiplier $M_\iota$ is contained in the next lemma:
\begin{lem} Consider the multipliers $M_\iota,\, \iota\in\{\kappa,\nu\}$. For two functions $f,g\in H^s(\Torus\times \rr), \, s> 1$, we have the following product rule
\begin{align}\label{M_product_rule_Hs}
\|M_\iota(fg)\|_{H^s}\leq &C\|M_\iota f\|_{H^s}\|M_\iota g\|_{H^s}.
\end{align}
Similarly, we have the following product rule for $A_\iota,\, \iota\in\{\kappa,\nu\},\, s> 1$,
\begin{align}\label{A_product_rule_Hs}
\|A_\iota (fg)\|_{L^2}\leq C\|A_\iota f\|_{L^2}\|A_\iota g\|_{L^2}.
\end{align}
\end{lem}

\begin{proof} First recall the product rule for the usual Sobolev functions on $\Torus\times \rr$:
\begin{align}
\|fg\|_{H^s}\leq C\|f\|_{H^s}\|g\|_{H^s},\quad s>1 .
\end{align}On the other hand, we recall that the bound \eqref{M_bound} yields that
\begin{align}
\|f\|_{H^s}\approx\|M_\iota f\|_{H^s}, \quad\|g\|_{H^s}\approx\|M_\iota g\|_{H^s},\, \iota\in{\{\kappa,\nu\}}.
\end{align}
Combining these estimates yields the result \eqref{M_product_rule_Hs}. 

To prove the product estimate \eqref{A_product_rule_Hs}, we observe the following relation for $x,y\geq 0,$
\begin{align}
|x+y|^{2/3}\leq |x|^{2/3}+|y|^{2/3}.
\end{align}
As a result, we have that
\begin{align}
e^{\delta \kappa^{1/3}|k|^{2/3}t}\leq e^{\delta \kappa^{1/3}|k-\ell|^{2/3}t}e^{\delta \kappa^{1/3}|\ell|^{2/3}t},\quad \forall k,\ell\in \mathbb{Z}.
\end{align}
Next we combine the above relation and the bounds of $M_\iota$ \eqref{M_bound} to derive the following,
\begin{align}
\|&A_\iota (f g)\|_{L^2}^2\\
\leq&\sum_{k\in \mathbb{Z}}\int e^{2\delta \kappa^{1/3}|k|^{2/3}t}M_{\iota}^2(k,\eta)\bigg|(1+|k|^2+|\eta|^2)^{ \frac{s}{2}}\sum_{\ell\in \mathbb{Z}}\int |\widehat{f}(k-\ell,\eta-\xi)|\ |\widehat {g}(\ell,\xi)| d\xi\bigg|^2 d\eta\\
\leq & C \sum_{k\in \mathbb{Z}}\int\bigg|\sum_{\ell\in \mathbb{Z}}\int \left(e^{\delta \kappa^{1/3}|k-\ell|^{2/3}t}M_\iota(k-\ell,\eta-\xi)(1+|k-\ell|^2+|\eta-\xi|^2)^{\frac{s}{2}}|\widehat{f}(k-\ell,\eta-\xi)|\right)\\
&\quad\quad\quad\quad\quad\quad\quad\quad\times\left( e^{\delta \kappa^{1/3}|\ell|^{2/3}t} |\widehat {g}(\ell,\xi)|\right) d\xi\bigg|^2 d\eta\\
&+ C \sum_{k\in \mathbb{Z}}\int\bigg|\sum_{\ell\in \mathbb{Z}}\int \left(e^{\delta \kappa^{1/3}|\ell|^{2/3}t}M_\iota(\ell,\xi)(1+| \ell|^2+|\xi|^2)^{\frac{s}{2}}|\widehat{g}( \ell,\xi)|\right)\left( e^{\delta \kappa^{1/3}|k-\ell|^{2/3}t} |\widehat {f}(k-\ell,\eta-\xi)|\right) d\xi\bigg|^2 d\eta.
\end{align}
Now we apply the Young's convolution inequality, H\"older's inequality and {the} inequality $\sum_{k\in \mathbb Z}\|\wh F_k(\cdot)\|_{L^1_\eta}\leq C(\sum_{k\in \mathbb Z}\|\lan\cdot\ran^s\wh F_k(\cdot)\|_{L^2_\eta}^2)^{1/2},\, s> 1$ to derive that
\begin{align}
\|A_\iota(fg)\|_{L^2}\leq &C\left(\|A_\iota f \|_{L^2}\sum_{k\in\mathbb{Z}}\left\| \myb{e^{\delta \kappa^{1/3}|k|^{2/3} t}\wh g_k(\cdot)}\right\|_{L^1_\eta}+\|A_\iota g\|_{L^2}\sum_{k\in\mathbb{Z}}\left\|e^{\delta\kappa^{1/3}|k|^{2/3}t}\wh f_k(\cdot)\right\|_{L^1_\eta}\right)\\
\leq& C\|A_\iota f\|_{L^2}\ \|A_\iota g\|_{L^2}.
\end{align}
This concludes the proof of the estimate \eqref{A_product_rule_Hs} and completes the proof of the lemma.
\end{proof}

In the proof, the following commutator estimate is needed:
\begin{lem}[Commutator estimates] The following commutator estimate concerning $M_\iota$ is satisfied
\begin{align}
|M_\iota (t,k,\eta)\lan k,\eta\ran^{s}-M_\iota(t,k,\xi)\lan k,\xi\ran^s |\leq C\frac{|\eta-\xi|}{|k|}(\myb{\lan \eta-\xi\ran^{s}}+\lan k,\xi\ran^{s}),\quad k \nq 0.\label{commutator_estimate}
\end{align}
\end{lem}
\begin{proof}

Here the difference $|M_\iota (t,k,\eta)\lan k,\eta\ran^{s}-M_\iota(t,k,\xi)\lan k,\xi\ran^s|$ can be decomposed as follows:
\begin{align}\label{Commutator_1_2}
\quad \quad M_\iota&(t,k,\eta)\lan k,\eta\ran^s-M_\iota(t,k,\xi)\lan k,\xi\ran^s\\
=&M_\iota(t,k,\eta)\left((1+k^2+\eta^2)^{s/2}-(1+k^2+\xi^2)^{s/2}\right)+(M_\iota(t,k,\eta)-M_\iota(t,k, \xi))(1+k^2+ \xi^2)^{s/2}\mathbf{1}_{ k\neq 0}\\
=:&\mathcal{T}_1+\mathcal{T}_2.
\end{align}
For the first term in \eqref{Commutator_1_2}, one applies the mean value theorem to obtain that there exists $\theta\in[0,1]$, such that the  following estimate holds
\begin{align}
|\mathcal{T}_1|=&\bigg|M_\iota(k,\eta)\frac{s}{2}(1+k^2+((1-\theta)\eta+\theta\xi)^2)^{\frac{s}{2}-1})2((1-\theta)\eta+\theta\xi)(\eta-\xi)\bigg|\\
\leq&C\bigg((1+k^2+ \xi^2)^{\frac{s-1}{2}}+(1+k^2+\eta^2)^{\frac{s-1}{2}}\bigg)|\eta-\xi|\myb{\leq C\frac{|\eta-\xi|}{|k|}(\lan k,\xi\ran^s+\lan  \eta-\xi\ran^s)}.
\end{align}
To estimate the $\mathcal{T}_2$ term in \eqref{Commutator_1_2}, we apply the property \eqref{M_property_common_pa_eta} and the mean value theorem to obtain that
\begin{align}
|\mathcal{T}_2|\leq\frac{C|\eta-\xi|}{|k|}(1+k^2+\xi^2)^{\frac{s}{2}}\mathbf{1}_{k\neq 0}.
\end{align}
Combining the two estimates and \eqref{Commutator_1_2}, we obtain \eqref{commutator_estimate}. 
The proof of the lemma is finished.
\end{proof}

\small\footnotesize
\bibliographystyle{abbrv}
\bibliography{nonlocal_eqns,JacobBib,SimingBib}

\end{document}